\documentclass[11pt]{amsart}
\usepackage[usenames]{color}
\usepackage{fullpage}
\usepackage{amscd}
\usepackage{amssymb, latexsym}
\usepackage[all,2cell,ps]{xy}
\usepackage{mathdots}

\theoremstyle{plain}
\newtheorem{thm}{Theorem}[section]
\newtheorem{theorem}[thm]{Theorem}

\newtheorem{corollary}[thm]{Corollary}

\newtheorem{lemma}[thm]{Lemma}

\newtheorem{proposition}[thm]{Proposition}
\newtheorem{ques}[thm]{Question}
\newtheorem{conjecture}[thm]{Conjecture}

\theoremstyle{definition}
\newtheorem{de}[thm]{Definition}

\newtheorem{remark}[thm]{Remark}
\newtheorem{convention}[thm]{Convention}
\newtheorem{example}[thm]{Example}

\newcommand{\Z}{\mathbb{Z}}
\newcommand{\N}{\mathbb{N}}

\newcommand{\cz}[1]{{\color{red}{#1}\color{black}{}}}

\newcommand{\id}{\mathrm{id}}

\newcommand{\aut}[1]{\mathrm{Aut}(#1)}
\newcommand{\dis}[1]{\mathrm{Dis}(#1)}

\newcommand{\Ker}{\mathop{\mathrm{Ker}}}
\newcommand{\kker}{\mathop{\mathrm{ker}}}

\newcommand{\Aut}{\mathop{\mathrm{Aut}}}

\newcommand{\Soc}{\mathop{\mathrm{Soc}}}

\newcommand{\obraz}[2]{\;\vphantom{x}^{#2}\!{#1}}
\newcommand{\obrazS}[2]{\vphantom{x}^{#2}\!{#1}}
\numberwithin{equation}{section}

\begin{document}
\title{Indecomposable involutive solutions of \\the Yang-Baxter equation of multipermutation level 2 \\ with non-abelian permutation group}

\author{P\v remysl Jedli\v cka}
\author{Agata Pilitowska}

\address{(P.J.) Department of Mathematics, Faculty of Engineering, Czech University of Life Sciences, Kam\'yck\'a 129, 16521 Praha 6, Czech Republic}
\address{(A.P.) Faculty of Mathematics and Information Science, Warsaw University of Technology, Koszykowa 75, 00-662 Warsaw, Poland}

\email{(P.J.) jedlickap@tf.czu.cz}
\email{(A.P.) agata.pilitowska@pw.edu.pl}


\date{\today}

\begin{abstract}
We give a complete characterization of all indecomposable involutive solutions of the Yang-Baxter equation of multipermutation level~2.
In the first step we present a construction of some family of such solutions and in the second step we prove that every indecomposable involutive solution of the Yang-Baxter equation with multipermutation level~2 is a homomorphic image of a solution previously constructed. Analyzing this epimorphism, we are able to obtain  all such solutions up to isomorphism and enumerate these of small sizes.
\end{abstract}

\maketitle


\section{Introduction}\label{sec:intr}

The Yang-Baxter equation is a fundamental equation occurring in mathematical physics. It appears, for example, in integrable models in statistical mechanics, quantum field theory or Hopf algebras~(see e.g. \cite{Jimbo, Kassel}). Searching for its solutions has been absorbing researchers for many years.

Let us recall that, for a vector space $V$, a {\em solution of the Yang--Baxter equation} is a linear mapping $r:V\otimes V\to V\otimes V$ 
 such that
\begin{align*}
(id\otimes r) (r\otimes id) (id\otimes r)=(r\otimes id) (id\otimes r) (r\otimes id).
\end{align*}

Description of all possible solutions seems to be extremely difficult and therefore
there were some simplifications introduced by Drinfeld in \cite{Dr90}.
Let  $X$ be a basis of the space $V$ and let $\sigma:X^2\to X$ and $\tau: X^2\to X$ be two mappings. We say that $(X,\sigma,\tau)$ is a {\em set-theoretic solution of the Yang--Baxter equation} if
the mapping 
$$x\otimes y \mapsto \sigma(x,y)\otimes \tau(x,y)$$ extends to a solution of the Yang--Baxter
equation. It means that $r\colon X^2\to X^2$, where $r=(\sigma,\tau)$,  is a bijection and
satisfies the \emph{braid relation}:
\begin{equation}\label{eq:braid}
(id\times r)(r\times id)(id\times r)=(r\times id)(id\times r)(r\times id).
\end{equation}

A solution is called {\em non-degenerate} if the mappings $\sigma_x=\sigma(x,\_)$ and $\tau_y=\tau(\_\,,y)$ are bijections,
for all $x,y\in X$.
A solution $(X,\sigma,\tau)$ is {\em involutive} if $r^2=\mathrm{id}_{X^2}$, and it is
\emph{square free} if $r(x,x)=(x,x)$, for every $x\in X$.

\begin{convention}

All solutions, we study in this paper, are set-theoretic, non-degenerate and involutive, so we will call them simply \emph{solutions}. 

\end{convention}

Various algebraic structures are naturally associated with solutions $(X,\sigma,\tau)$.  Among others, there is a group of bijections defined on the set $X$. \emph{The permutation group} $\mathcal{G}(X)=\left\langle \sigma_x\colon x\in X\right\rangle$ of a solution  is the subgroup of the symmetric group $S(X)$ generated by mappings $\sigma_x$, with $x\in X$. The group $\mathcal{G}(X)$ is also called \emph{the Yang-Baxter group} (IYB group) associated to the solution $(X,\sigma,\tau)$. A solution $(X,\sigma,\tau)$ is \emph{indecomposable} if the permutation group  $\mathcal{G}(X)$ acts transitively on $X$.

A solution with regular permutation group is called \emph{uniconnected} (see \cite{Rump20}). By definition, uniconnected solutions are indecomposable.

In \cite[Section 3.2]{ESS} Etingof, Schedler and Soloviev introduced, for each solution $(X,\sigma,\tau)$, the equivalence relation $\sim$ on the set $X$: for each $x,y\in X$
\[
x\sim y\quad \Leftrightarrow\quad \tau(\_\,,x)=\tau(\_\,,y).
\]
They showed that the quotient set $X/\mathord{\sim}$ can be again endowed
with a structure of a solution and they call such a solution the {\em retraction} of the solution~$(X,\sigma,\tau)$ and denote it by
$\mathrm{Ret}(X)$. A solution~$(X,\sigma,\tau)$ is said to be a {\em permutation} solution (or {\em multipermutation solution of level}~$1$) if $|\mathrm{Ret}(X)|=1$. 
A non-permutation solution~$(X,\sigma,\tau)$ is said to be a
{\em multipermutation solution of level}~$2$ (or briefly, $2$-\emph{permutational} solution), if $|\mathrm{Ret}(X)/\mathord{\sim}|=1$.

Even
describing all set-theoretic solutions still remains a very difficult task. For this reason it looks fruitful to prospect less complex solutions and, in some sense, treat them as ``bricks'' to build new ones. An approach to study solutions is to classify them according to their permutation groups. Special attention is now paid when such groups are transitive. Clearly, knowledge of indecomposable solutions could allow one to build all the decomposable ones. For example, Rump showed in \cite{Rump05} that every finite square free solution admits a nontrivial decomposition.

Cedó and Okniński characterized in \cite{CO21} so called {\em simple} solutions (solutions with only trivial congruences). In particular, they proved that each finite simple solution (not of a prime order) is indecomposable. Moreover, by the results from \cite{CCP}, every finite indecomposable solution is a dynamical extension of a simple solution.  

Unfortunately, it is not easy to find, in general, all the indecomposable solutions. The investigation of indecomposable solutions was initiated  by Etingof et. al. in \cite{ESS}. They showed there that such solutions of prime cardinality are permutation solutions with cyclic permutation group. In \cite{JPZ22} the authors of this paper together with Zamojska-Dzienio gave a complete system of three invariants for finite non-isomorphic  indecomposable solutions with cyclic permutation groups (cocyclic solutions). Smoktunowicz and Smoktunowicz presented in \cite{SS18} a construction of all finite indecomposable solutions based on one-generated left braces. Also Rump in \cite{Rump20} gave another method of constructing indecomposable solutions from left braces. Additionally, the results of Bachiller, Cedó and Jespers~\cite{BCJ16} allow one to construct all solutions with a given permutation group. 
But all the above constructions require being able to construct all (finite) left braces. 
Moreover, by results of Smoktunowicz and Smoktunowicz, we know that one-generated  left braces, and in consequence, indecomposable solutions of non-prime cardinality are quite frequent. 

In the light of what we know it is reasonable
to focus on a smaller class of indecomposable solutions, for instance multipermutation solutions of level $2$ with non-abelian permutation group. Note that study multipermutation solutions of low levels attracts, among others, researchers interested in the quantum spaces (see e.g. \cite{GIM11}).

By results of Gateva-Ivanova, multipermutation solutions of level $2$ can be characterized in an easy way.
\begin{theorem}\cite[Proposition 4.7]{GI18}\label{GI:mper}
Let $(X,\sigma,\tau)$ be a solution and $|X|\geq 2$. Then
$(X,\sigma,\tau)$ is a multipermutation solution of level $2$ if and only if the following condition holds for every $x,y,z\in X$:
\begin{align}\label{eq:2per}
\sigma_{\sigma_y(x)}=\sigma_{\sigma_z(x)}.
\end{align}

\end{theorem}


All finite indecomposable solutions of the Yang-Baxter equation of multipermutation level at most 2 with abelian permutation groups were constructed in~\cite{JPZ21}.
Probably the first example of an indecomposable solution of multipermutation level $2$  with non-abelian permutation group was presented by 
Etingof et. al. in \cite{ESS}. In \cite[Subsection 3.3.]{ESS} they constructed all indecomposable solutions of multipermutation level $2$ of order $6$. Later on a new family of such solutions of order $2n$, for arbitrary $n\in \mathbb{N}$ and with  the permutation group isomorphic to $\Z_n\rtimes \Z_2$, was presented in \cite[Example 2.6]{JPZ21}. Just recently, Castelli characterized in \cite{C21} all the uniconnected 
solutions having odd size, a $Z$-group permutation group and having multipermutation level at most 2. He classified all such solutions with odd square-free order, up to isomorphism \cite[Theorem 5.2]{C21}.

In our paper, we study general indecomposable solutions
of multipermutation level~2. In \cite{JPZ20}
the authors, together with Zamojska-Dzienio, presented a way how to obtain arbitrary solutions of multipermutation level~2. Now we apply this method to describe indecomposable ones. We show that an important r\^ole is played by one normal abelian subgroup of the permutation group, called the {\em displacement group}.
We also find a way how to obtain (in theory) all indecomposable solutions  
of multipermutation level~2 of arbitrary size $n$ and we show that there are at least $2^{n/2}-1$ such solutions when~$n$ is a power of~$2$.
This result has not been expected since the numbers
of indecomposable solutions is known by~\cite{AMV} only up to the size of~$11$ and we have not have enough evidence
about the growth of the numbers for powers of primes.

The paper is organized as follows. 
Since the most standard way how to construct solutions is using braces, in Section \ref{sec:unicon} we recall that every uniconnected
solution can be constructed from a left brace on the basis of results of Rump and we present some examples of this construction. In Section~\ref{sec:mul} we recall
a connection between solutions of multipermutation level~$2$
and so-called {\em $2$-reductive} solutions
and, based on this relationship, we describe the structure
of the displacement group.
Section \ref{sec:constr} contains a construction of a special class of indecomposable  solutions of multipermutation level 2 (Theorem \ref{th:main}).
We study properties of the construction. In particular we show that the permutation group of each such solution is a semidirect product of its displacement group and a cyclic group and the automorphism group of the solution is transitive.
Finally, in Section~\ref{sec:homo} we characterize all congruence relations of solutions constructed in Section \ref{sec:constr}, we show that
every indecomposable solution of multipermutation
level~$2$ is an image of one particular solution
and also give a characterization of all the solutions that can be obtained by the construction described in Theorem \ref{th:main}.
This allows us to calculate all such solutions up to the size of~$17$, which is done in Section~6. Additionally, we prove that the automorphism group of any indecomposable solution of multipermutation
level~$2$ is regular.

\section{Uniconnected $2$-permutational solutions}\label{sec:unicon}

In this section we recall some facts about braces and uniconnected solutions.
 Rump in \cite[Definition 2]{Rump07} introduced the notion of a brace and showed the correspondence between such structures and solutions. Here we use an equivalent definition formulated by Ced\'{o}, Jespers and Okni\'{n}ski.

\begin{de}\cite[Definition 2.2]{CJO}
 An algebra $(B,+,\circ)$ is called a {\em left brace} if $(B,+)$ is
 an abelian group, $(B,\circ)$ is a group and the operations satisfy, for all $a,b,c\in B$,
 \begin{equation}\label{lb}
  a\circ b+a\circ c = a\circ(b+c)+a.
 \end{equation}
 \end{de}
 \noindent
The inverse element to~$a$ in the group $(B,\circ)$ shall be denoted by $a^{-}$.

For a left brace $(B,+,\circ)$ there is an action $\lambda\colon(B,\circ)\rightarrow\Aut{(B,+)}$, where $\lambda(a)=\lambda_a$ and for all $b\in B$, $\lambda_a(b)=a\circ b-a$. The kernel of the group homomorphism $\lambda$ is called the \emph{socle} of the left brace $(B,+,\circ)$ and denoted by $\Soc(B)$, i.e.
\begin{equation*}
\Soc(B)=\Ker(\lambda)=\{a\in B\colon \lambda_a=\id\}=\{a\in B\colon a\circ b=a+b \text{ for all } b\in B\}.
\end{equation*}
$\Soc(B)$ is a normal subgroup of both: the group $(B,+)$ and $(B,\circ)$. Hence the quotient $(B,\circ)/\Soc(B)$ of the multiplicative group is also the quotient of the additive group $(B,+)$ and the factor left brace $B/\Soc(B):=(B,+,\circ)/\Soc(B)$ by $\Soc(B)$ is well defined.

A brace is called {\em cyclic} if the additive group $(A,+)$ is cyclic, {\em bicyclic} if both groups are cyclic and it is {\em trivial} if $\Soc(B)=B$.

The \emph{associated solution} $(B,\sigma,\tau)$ to a left brace $(B,+,\circ)$ is defined on a set $B$ as follows: $\sigma_x(y):=\lambda_x(y)$ and $\tau_y(x):=\lambda^{-1}_{\lambda_x(y)}(x)$ for $x,y\in B$.

\vskip 2mm
One of methods to construct braces from given ones is by their semidirect product. Let $(B_1,+_1,\circ_1)$ and 
$(B_2,+_2,\circ_2)$ be two left braces and let $\alpha\colon (B_2,\circ_2)\to \Aut{(B_1,+_1,\circ_1)}$ be a homomorphism of groups. The brace $B_1\rtimes_{\alpha} B_2=(B_1\times B_2,+,\circ)$ with  $(B_1\times B_2,+)=(B_1,+_1)\times (B_2,+_2)$ and for $(x_1,x_2),(y_1,y_2)\in B_1\times B_2$ 
\[
(x_1,x_2)\circ(y_1,y_2)=(x_1\circ_1 \alpha(x_2)(y_1),x_2\circ_2 y_2)
\]
is called the {\em semidirect product of the left braces} with respect to $\alpha$. In particular,
\[
\lambda_{(x_1,x_2)}((y_1,y_2))=
(x_1\circ_1 \alpha(x_2)(y_1)-x_1,x_2\circ_2 y_2-x_2).
\]

A subset $X$ of a left brace $(B,+,\circ)$ is called a {\em cycle base} if $X$ is a union of orbits of the action $\lambda$ and generates the additive group $(B,+)$. A cycle base is {\em transitive} if it is a single orbit. 
Note that a transitive cycle base always exists in a cycle brace.

In \cite[Theorem 4.2]{Rump20} Rump showed that with a left brace $(B,+,\circ)$ with a transitive cycle base $X$ one can associate a uniconnected solution $(B,\sigma,\tau)$: 
\begin{align}\label{sol:un}
\sigma_x(y):=(\lambda_x(g))^{-}\circ y
\end{align}
for $x,y\in B$ and $g\in X$, such that the permutation group  $\mathcal{G}(B)$ is isomorphic to the group $(B,\circ)$.  
On the other hand, each uniconnected solution can be constructed in this way. 

We now present an example of braces which later on will be good illustration of construction described in Section \ref{sec:constr}. As usual, we will denote by $\Z_n$ the group $(\Z_n,+_n)=\mathbb{Z}/(n)$.
\begin{example}\label{exm:unisol}
Let $n\in \mathbb{N}$, $(G,+,0)$ be an abelian group and $\alpha\in\aut{G}$ be of an order dividing~$n$.
Further, let $(B,+,\circ)=(G\times \Z_n,+,\circ)$ be the $\alpha$-semidirect product of two trivial braces: $(G,+,+)$ and $(\Z_n,+_n,+_n)$. 
In this case, for $(a,i),(b,j)\in G\times \Z_n$ we have
\[
(a,i)\circ (b,j)=(a+\alpha^i(b),i+_n j),
\quad
(a,i)^{-}=(-\alpha^{-i}(a),-i),\quad{\rm and}\quad \lambda_{(a,i)}((b,j))=(\alpha^i(b),j).
\]
Assume now that an element $(g,h)\in G\times \Z_n$ lies in a transitive cycle base of $(B,+,\circ)$ (if such transitive cycle base exists). 
By Rump's result, $(G\times \Z_n,\sigma, \tau)$ with 
\[
\sigma_{(a,i)}((b,j))=(\lambda_{(a,i)}((g,h)))^{-}\circ (b,j)=
(-\alpha^{i-h}(g),-h)\circ (b,j)
=(\alpha^{-h}(b)-\alpha^{i-h}(g),j-h),
\]
for $(a,i),(b,j)\in G\times \Z_n$, is a uniconnected solution with the permutation group  $\mathcal{G}(G\times \Z_n)$ isomorphic to the semidirect product $G\rtimes \Z_n$ of groups $(G,+,0)$ and $\Z_n$. 

Moreover,
\[
 \sigma_{(a,i)}^{-1}((b,j))=(\alpha^h(b)+\alpha^i(g),j+h).
\]

It is worth noticing that the solution satisfies~\eqref{GI:mper};
for $(a,i),(a',i'), (b,j), (c,k)\in  G\times \Z_n$ we have:
\[
\sigma^{-1}_{\sigma^{-1}_{(a,i)}((b,j))}((c,k))=\sigma^{-1}_{(\alpha^h(b)+\alpha^i(g),j+h)}((c,k))=(\alpha^h(c)+\alpha^{j+h}(g),k+h)=\sigma^{-1}_{\sigma^{-1}_{(a',i')}((b,j))}((c,k)).
\]
Hence, the solution is ~$2$-permutational.


\end{example}

\begin{example}\label{exm:uni223}
Let $(G,+,0)=\Z_2\times \Z_2$ and 
let $\alpha=\left(\begin{smallmatrix}
0&1\\
1&1\end{smallmatrix}\right)\in\Aut{\Z_2\times \Z_2}$ be an automorphisms 
of order $3$. 
Let $n=3$ and $(g,h)=((1,0),1)$. Clearly,
the set $\{(\alpha^i((1,0)),1)\mid i\in\Z_3\}=\{((1,0),1), ((0,1),1),((1,1),1)$ generates the group $\Z_2\times \Z_2\times \Z_3$. Hence, by Example \ref{exm:unisol}, $(\Z_2\times \Z_2\times \Z_3,\sigma,\tau)$, with
 \[
\sigma_{((a_1,a_2),i)}(((b_1,b_2),j))=(\alpha^{-1}((b_1,b_2))-\alpha^{i-1}((1,0)),j-1)=((b_1+b_2,b_1)-\alpha^{i-1}((1,0)),j-1),
\]
is a uniconnected ~$2$-permutational solution with the permutation group  $\mathcal{G}(\Z_2\times \Z_2\times \Z_3)$ isomorphic to $(\Z_2\times \Z_2)\rtimes \Z_3$.

Straightforward calculations show that all permutations $\sigma_x$ are composed of cycles of length $3$. We have also the abelian subgroup $(\{\id, \sigma_{\mathbf{1}}\sigma^{-1}_{\mathbf{0}},\sigma_{\mathbf{2}}\sigma^{-1}_{\mathbf{0}},\sigma_{\mathbf{1}}\sigma^{-1}_{\mathbf{0}}\sigma_{\mathbf{2}}\sigma^{-1}_{\mathbf{0}}\},\circ)$ of $\mathcal{G}(\Z_2\times \Z_2\times \Z_3)$ with $\mathbf{0}=((0,0),0)$, $\mathbf{1}=((0,0),1)$ and  $\mathbf{2}=((0,0),2)$,  which is isomorphic to the group $\Z_2\times \Z_2$ and $\langle\sigma_{\mathbf{0}}\rangle\cong \Z_3$.

\end{example}

\begin{example}\label{exm:Rump}
According to Rump's classification of cyclic left braces \cite{Rump07B}, for a left cyclic brace $(B,+,\circ)$ with $|B|=p^m$, for some prime $p$ and $m\in \mathbb{N}$, the adjoint group $(B,\circ)$ admits a cyclic subgroup of index less or equal to $2$. Hence, if $p\neq 2$, the group $(B,\circ)$ is cyclic, and consequently $(B,+,\circ)$ is bicyclic.

In \cite[Corollary 4.8.]{JPZ21} we, together with Zamojska-Dzienio, have already described all finite indecomposable solutions of the multipermutation level at most $2$ with abelian (cyclic) permutation group. Since an abelian group acting transitively is regular, all of them are uniconnected.

For a non-abelian cyclic left brace $(B,+,\circ)$, with $|B|=2^m$, either $\Soc(B)$ is of index $2$ in $(B,+,\circ)$ or $B/\Soc(B)$ is not bicyclic. In the latter case, the uniconnected solution associated with $(B,+,\circ)$  is not a multipermutation solution of level at most $2$. So, by \cite[Proposition 12]{Rump07B}, to construct all uniconnected solutions of the multipermutation level at most $2$, of size $2^m$ with $m>2$, and originated from a cyclic left brace, we have to consider only two types of such left braces of size~$2^m$. 
\vskip 3mm
\noindent
{\bf Case 1.} Let $(\Z_{2^m},+,\circ)$ be a cyclic left brace with the adjoint group being a dihedral group. Then, for $a,b\in \Z_{2^m}$,
\[
a\circ b=a+(-1)^ab=\begin{cases}a+b, \quad{\rm if}\; a\; {\rm is\; even,}\\
a-b, \quad{\rm if}\; a\; {\rm is\; odd.}
\end{cases}
\]
Let $g=1$. Hence $\lambda_a(1 )=a\circ 1-a=(-1)^a=(\lambda_a(1))^{-}$. 
Thus $(\Z_{2^m},\sigma,\tau)$, with
\[
\sigma_a(b)=(\lambda_a(1))^{-}\circ b=
(-1)^a+(-1)^{{(-1)}^a}b=
\begin{cases}1-b, \quad{\rm if}\; a\; {\rm is\; even,}\\
-1-b=(2^m-1)(1+b), \quad{\rm if}\; a\; {\rm is\; odd}
\end{cases}
\]
is a uniconnected solution with the permutation group  $\mathcal{G}(\Z_{2^m})$ isomorphic to the dihedral group $D_{2^m}$. Since the parity of both $1-b$ and $-1-b$ is the same for arbitrary $b\in \Z_{2^m}$, $(\Z_{2^m},\sigma,\tau)$ is a multipermutation solution of level $2$.  

\smallskip

\noindent
{\bf Case 2.} Let $(\Z_{2^m},+,\circ)$ be a cyclic left brace with the adjoint group being a generalized quaternion group. Then, for $a,b\in B$,
\[
a\circ b=a+(2^{m-1}-1)^ab=\begin{cases}a+b, \quad{\rm if}\; a\; {\rm is\; even,}\\
a+(2^{m-1}-1)b, \quad{\rm if}\; a\; {\rm is\; odd.}
\end{cases}
\]
For $g=1$, 
\[
\lambda_a(1)=a+(2^{m-1}-1)^a-a=\begin{cases}1, \quad{\rm if}\; a\; {\rm is\; even,}\\
2^{m-1}-1, \quad{\rm if}\; a\; {\rm is\; odd.}
\end{cases}
\] 
The uniconnected solution $(\Z_{2^m},\sigma,\tau)$ is defined in the following way:
\[
\sigma_a(b)=(\lambda_a(1))^{-}\circ b=\begin{cases}(1-2^{m-1})(1-b), \quad{\rm if}\; a\; {\rm is\; even,}\\
-1+(2^{m-1}-1)b, \quad{\rm if}\; a\; {\rm is\; odd.}
\end{cases}
\]
As in the previous case, it is $2$-permutational and its permutation group is isomorphic to the generalized quaternion group $Q_{2^m}$.
\end{example}

In Section \ref{sec:constr} we will show that solutions described in Examples \ref{exm:unisol} - \ref{exm:Rump} 
are samples of a more general construction.

\section{Multipermutation solutions of level $2$}\label{sec:mul}
In \cite{JPZ20} the authors, together with Zamojska-Dzienio, presented a method of constructing all the solutions of multipermutation level 2. These solutions fall into two classes: $2$-reductive ones and not $2$-reductive ones. The former ones can be effectively constructed using a set of abelian groups and a family of constants. The non-$2$-reductive solutions can be also easily constructed, using a $2$-reductive solution and a permutation. Let us shortly remind these constructions.

The less complicated class is formed by so called 2-reductive solutions.
Recall that a solution $(X,\sigma,\tau)$ is called {\em $2$-reductive}, if it satisfies, for all $x,y\in X$,
\begin{equation}
 \sigma_{\sigma_y(x)}=\sigma_x. \label{eq:2red}
\end{equation}
By \cite[Lemma 2.7]{JPZ20}, Condition \eqref{eq:2red} is equivalent to the following one:
 \begin{equation}
 \sigma_{\sigma^{-1}_y(x)}=\sigma_x. \label{eq:2red1}
\end{equation}
By \cite[Theorem 7.7]{JPZ20} each $2$-reductive solution is a multipermutation solution of level $2$ and note also (\cite[Proposition 8.2]{GIC12}, \cite[Proposition 4.7]{GI18}) that a square free solution is of level at most $2$ if and only if it is $2$-reductive.

The permutations of the $2$-reductive solutions will be in this article, for some reason explained below,
denoted by~$L_x$ and $\mathbf{R}_x$.
These solutions can be completely characterized as follows:

\begin{theorem}\label{thm:2red}\cite[Theorem 7.8]{JPZ20}
Each $2$-reductive solution ~$(X,L,\mathbf{R})$ is a disjoint union, over a set $I$, of abelian groups
$(A_j,+)=\left<\{c_{i,j}\mid i\in I\}\right>$, for every $j\in I$, with
\begin{equation}\label{eq:7.8}
L_x(y)=y+c_{i,j}\in A_j\quad {\rm and} \quad \mathbf{R}_y(x)= x-c_{j,i}\in A_i,
\end{equation}
where $x\in A_i$ and $y\in A_j$.
\end{theorem}

According to \cite[Lemma 3.3]{JPZ20}, for each $2$-reductive solution the group $\langle L_x\mid x\in X\rangle$ is always abelian.
Furthermore, by  \cite[Theorem 6.12]{JPZ20} each multipermutation solution of level $2$ originates from some $2$-reductive one. The procedure is the following.

\begin{theorem} \label{th:2per}\cite[Algorithm 7.13]{JPZ20}
Let $(X,L,\mathbf{R})$ be a $2$-reductive solution. Assume that there exists $a\in X$ with $L_a=\mathrm{id}$ and let $\pi$ be a permutation of the set $X$ satisfying, for $x,y\in X$, the condition:
\begin{align}\label{pi}
L_{\pi(y)}\pi L_x=L_{\pi(x)}\pi L_y.
\end{align}
Then $(X,\sigma,\tau)$ with $\sigma_x=L_x\pi$ and $\tau_y=\pi^{-1}\mathbf{R}_{\pi(y)}$ is 
a multipermutation solution of level $2$.
\end{theorem}

On the other hand, each multipermutation solution of level $2$ defines a $2$-reductive one.
\begin{theorem}\cite[Theorem 7.12]{JPZ20} \label{th:2persol}
Let $(X,\sigma,\tau)$ be a multipermutation solution of level~$2$
and $e\in X$. Then $(X,L,\mathbf{R})$, where $L_x=\sigma_x\sigma_e^{-1}$ and $\mathbf{R}_y=\sigma_e\tau_{\sigma_e^{-1}(y)}$, for $x,y\in X$, is a $2$-reductive solution.
\end{theorem}
\noindent
We call the solution $(X,L,\mathbf{R})$ from Theorem \ref{th:2persol} the $\sigma_e^{-1}$-\emph{isotope} of $(X,\sigma,\tau)$. 
From now on, we shall assume that one such 2-reductive solution is always associated to our solution of multipermutation level 2.
Moreover, since $L_e=\mathrm{id}$, by \cite[Theorem 6.12]{JPZ20}, the solution $(X,L,\mathbf{R})$ satisfies Condition \eqref{pi} for $\pi=\sigma_e$, that
means $\sigma_x=L_x\pi=L_x\sigma_e$, for each~$x\in X$.

\subsection*{Displacement group}


Note that, for every $x,y\in X$,
\begin{equation}\label{dis}
L_xL_y^{-1}=L_x\pi\pi^{-1}L_y^{-1}=\sigma_x\sigma_y^{-1}\qquad
\text{ and }\qquad
L_xL_e^{-1}=L_x.
\end{equation}
This implies that, for a multipermutation solution of level~$2$, the group
$$\langle\sigma_x\sigma_y^{-1}\mid x,y\in X\rangle=\langle L_xL_y^{-1}\mid x,y\in X\rangle
=\langle L_x\mid x\in X\rangle$$
is an abelian subgroup of $\mathcal{G}(X)$.

Such groups appear, for example, in the theory of racks and quandles under names: \emph{displacement groups} \cite{HSV} or  \emph{transvection groups} \cite{J82}. They play important r\^ole in the theory  since some properties of racks or quandles may be described by properties of these groups. 
In Yang-Baxter equation solutions theory these groups 
have gained very little attention so far. A  probable reason is that,
given a left brace~$(B,+,\circ)$, we have $\lambda_{0}=\mathrm{id}$ and
therefore $\mathcal{G}(B)=\langle\sigma_x\sigma_y^{-1}\mid x,y\in B\rangle$, for the associated solution~$(B,\sigma,\tau)$.
Nevertheless, it turns out that
these groups are crucial when
studying indecomposable multipermutation solutions of level~2
 (see e.g. Proposition \ref{prop:510}).  
Let us then denote the group by $\dis{X}$ and call it \emph{the displacement group of the solution}  $(X,\sigma,\tau)$.

For the rest of the section we fix the following notation: let us have a solution $(X,\sigma,\tau)$
which is a multipermutation solution of level~2 and fix an element $\widetilde 0\in X$. The choice of the element~$\widetilde 0$ might be of some relevance for decomposable solutions but
does not matter at all for indecomposable solutions because, as we shall see in~Proposition~\ref{prop:Auttrans}, the automorphism group of such solutions is transitive.
Let us denote $\pi=\sigma_{\widetilde 0}$ and let
$(X,L,\mathbf{R})$ be the $\pi^{-1}$-isotope defined in Theorem \ref{th:2persol}. We also, for
each~$i\in\Z$, denote $\widetilde i=\pi^i(\widetilde 0)$. The set~$\{\widetilde i\mid i\in\Z\}$ is a subset of~$X$
only but it turns out that it may be an important subset of~$X$. We also denote $D_x=L_xL_{\pi^{-1}(x)}^{-1}$, for each~$x\in X$.

\begin{lemma}
Let  $(X,\sigma,\tau)$ be a multipermutation solution of level $2$. For each~$x\in X$ and $i\in\Z$, we have
\begin{align}
\pi L_x&=L_{\widetilde 1}^{-1}L_{\pi(x)}\pi\label{piLx}\\
L_x\pi&=\pi L_{\pi^{-1}(x)}L_{\widetilde {-1}}^{-1}\label{Lxpi}\\
\pi^i L_x \pi^{-i}&=L_{\widetilde i}^{-1} L_{\pi^i(x)} \label{piLxpi}\\
\pi D_x&=D_{\pi(x)}\pi\label{piDx}\\
\sigma_{\widetilde i}&=\sigma_{\widetilde 1}^i\sigma_{\widetilde 0}^{1-i}\label{sigma_i}\\
\sigma_{\alpha(x)}&=\sigma_{x},\text{ for all }\alpha\in\dis{X}.\label{eq:disX}
\end{align}
\end{lemma}

\begin{proof}
 Using Equation~\eqref{pi} with $y=\widetilde 0$ we get
 \[ L_{\pi(x)}\pi L_{\widetilde 0}=L_{\widetilde 1}\pi L_x\]
 and since~$L_{\widetilde 0}=\mathrm{id}$, we obtain~\eqref{piLx}. Using Equation~\eqref{pi} with $x=\widetilde {-1}$ and $y=\pi^{-1}(x)$ we get
 \[ L_{\widetilde 0}\pi L_{\pi^{-1}(x)}=L_{x}\pi L_{\widetilde {-1}}\]
 and thus we obtain~\eqref{Lxpi}. In particular, by~\eqref{piLx} and ~\eqref{Lxpi} we have  \[
 \pi L_x^{-1}=L_{\pi(x)}^{-1}L_{\widetilde{1}}\pi.
 \]
The equation~\eqref{piLxpi} is trivial for~$i=0$, then, by induction and commutativity of the group $\dis{X}$,
 \begin{align*}
 \pi \pi^i L_x\pi^{-i} \pi^{-1} &= \pi L_{\widetilde i}^{-1} L_{\pi^i(x)}\pi^{-1}
  \stackrel{\eqref{piLx}}{=}L_{\widetilde{i+1}}^{-1}L_{\widetilde 1}\pi L_{\pi^i(x)}\pi^{-1}  \stackrel{\eqref{piLx}}{=}
 L_{\widetilde{i+1}}^{-1}L_{\widetilde 1}L_{\widetilde 1}^{-1} L_{\pi^{i+1}(x)}=
  L_{\widetilde{i+1}}^{-1}L_{\pi^{i+1}(x)},\\
  \pi^{-1} \pi^i L_x\pi^{-i} \pi &= \pi^{-1} L_{\widetilde i}^{-1} 
  L_{\pi^i(x)}\pi\stackrel{\eqref{Lxpi}}{=} 
 L_{\widetilde{-1}}L^{-1}_{\widetilde{i-1}}\pi^{-1}L_{\pi^i(x)}\pi\stackrel{\eqref{Lxpi}}{= }  
 L_{\widetilde{-1}}L_{\widetilde{i-1}}^{-1}L_{\pi^{i-1}(x)}L_{\widetilde{-1}}^{-1}
   =L_{\widetilde{i-1}}^{-1}L_{\pi^{i-1}(x)}.
  \end{align*}
Further by commutativity of the group 
 $\dis{X}$,
 \begin{multline*}
 \pi D_x=\pi L_x L_{\pi^{-1}(x)}^{-1}\stackrel{\eqref{piLx}}{=}
 L_{\pi(x)}L_{\widetilde 1}^{-1}\pi L_{\pi^{-1}(x)}^{-1}=
 L_{\pi(x)}L_{\widetilde 1}^{-1}  (L_{\pi^{-1}(x)}\pi^{-1})^{-1}
 \stackrel{\eqref{piLx}}{=}\\
 L_{\pi(x)}L_{\widetilde 1}^{-1}
 (\pi^{-1}L_{x}L_{\widetilde 1}^{-1})^{-1}=
 L_{\pi(x)}L_{\widetilde 1}^{-1} L_{\widetilde 1}
 L_{x}^{-1} \pi=D_{\pi(x)}\pi. 
 \end{multline*}
 Then, by an induction on~$i$,
 \begin{align*}
 \sigma_{\widetilde{i+1}}&=L_{\widetilde{i+1}}\pi\stackrel{\eqref{piLx}}{=}
  L_{\widetilde 1}\pi L_{\widetilde i}=\sigma_{\widetilde 1}\sigma_{\widetilde i}\pi^{-1}=
  \sigma_{\widetilde 1}\sigma^i_{\widetilde 1}\sigma_{\widetilde 0}^{1-i}\sigma_{\widetilde 0}^{-1}
  =\sigma^{i+1}_{\widetilde 1}\sigma_{\widetilde 0}^{-i},
  \end{align*}
 and
  \begin{align*}
  \sigma_{\widetilde{i-1}}&=L_{\widetilde{i-1}}\pi\stackrel{\eqref{piLx}}{=}
  \pi^{-1}L_{\widetilde 1}^{-1}L_{\widetilde i}\pi^2=\sigma_{\widetilde 1}^{-1}\sigma_{\widetilde i}\sigma_{\widetilde 0}=\sigma_{\widetilde 1}^{-1}\sigma^i_{\widetilde 1}\sigma_{\widetilde 0}^{1-i}\sigma_{\widetilde 0} 
  =\sigma^{i-1}_{\widetilde 1}\sigma_{\widetilde 0}^{2-i}.
 \end{align*}
 Finally, for $\alpha=\prod{L^{\varepsilon_j}_j}$, where $\varepsilon_j\in \{-1,1\}$, we obtain $\sigma_{\alpha(x)}=L_{\prod{L^{\varepsilon_j}_j}(x)}\pi\stackrel{\eqref{eq:2red}+\eqref{eq:2red1}}{=}L_x\pi=\sigma_x$.
\end{proof}


\begin{proposition}\label{prop:normal}
Let  $(X,\sigma,\tau)$ be a multipermutation solution of level $2$. Then $\dis{X}$ is a normal abelian subgroup of~$\mathcal{G}(X)$ and, for each $e\in X$,
$$\mathcal{G}(X)=\langle\dis{X}\cup\{\sigma_e\}\rangle.$$
Moreover,
\[\dis{X}=\langle \sigma_x\sigma_{e}^{-1}\mid x\in X\rangle=\langle L_x\mid x\in X\rangle=\left\{ \prod \sigma_{x_i}^{\varepsilon_i} \mid \sum \varepsilon_i=0\right\}.
\]


\end{proposition}

\begin{proof}
For $x\in X$, we have 
\[\sigma_x=\sigma_x\sigma_{e}^{-1}\sigma_e\in\langle\dis{X}\cup\{\sigma_e\}\rangle.\]
Hence $\mathcal{G}(X)=\langle\dis{X}\cup\{\sigma_e\}\rangle$ and $\dis{X}$ is a normal subgroup by Equation \eqref{piLxpi}.

Now $\sigma_x\sigma_y^{-1}=\sigma_x\sigma_e^{-1}(\sigma_y\sigma_e^{-1})^{-1}$ and therefore $\dis{X}$ is generated by the set $\{\sigma_x\sigma_e^{-1}\mid x\in X\}$. Setting $e=\widetilde 0$ we obtain $\{\sigma_x\sigma_e^{-1}\mid x\in X\}=\{L_x\mid x\in X\}$.

Let $A=\left\{ \prod \sigma_{x_i}^{\varepsilon_i} \mid \sum \varepsilon_i=0\right\}$. By \eqref{dis}, generators of $\dis{X}$ belong to $A$, so we have $\dis{X}\subseteq A$. To prove the converse inclusion, 
consider  $\alpha\in A$ written
in the following form:
\begin{align*}
\alpha=\sigma_{x_1}^{\varepsilon_1}\sigma_{x_2}^{\varepsilon_2}\ldots \sigma_{x_{n-1}}^{\varepsilon_{n-1}}\sigma_{x_n}^{\varepsilon_n}
\end{align*}
with $\sum \varepsilon_i=0$ and $n$ being an even natural number. 

We will prove by induction on $n$ that $\alpha\in\dis{X}$. Clearly, if $n=0$ then $\alpha=\id$, hence suppose $n\geq 2$.

If $\varepsilon_1=\varepsilon_n$ then there is $1<m<n$ such that $\sum_{i<m} \varepsilon_i=0$ and $\sum_{i\geq m} \varepsilon_i=0$. Let $\beta=\sigma_{x_1}^{\varepsilon_1}\ldots \sigma_{x_{m-1}}^{\varepsilon_{m-1}}$ and $\gamma=\sigma_{x_m}^{\varepsilon_m}\ldots\sigma_{x_n}^{\varepsilon_n}$. By the induction hypothesis, $\beta,\gamma\in \dis{X}$, and in consequence $\alpha=\beta\gamma\in\dis{X}$.

If $\varepsilon_1=-\varepsilon_n$
then $\alpha=\sigma_{x_1}^{\varepsilon_1}\beta\sigma_{x_n}^{-\varepsilon_1}$, where
$\beta=\sigma_{x_2}^{\varepsilon_2}\ldots \sigma_{x_{n-1}}^{\varepsilon_{n-1}}$. 
Since $\sum_{2\leq i\leq n-1} \varepsilon_i=0$, once again by the induction hypothesis, $\beta\in \dis{X}$. 
Moreover, 
$\alpha=\sigma_{x_1}^{\varepsilon_1}\beta\sigma_{x_n}^{-\varepsilon_1}=
\sigma_{x_1}^{\varepsilon_1}\beta\sigma_{x_1}^{-\varepsilon_1}\sigma_{x_1}^{\varepsilon_1}\sigma_{x_n}^{-\varepsilon_1}$
and both $\sigma_{x_1}^{\varepsilon_1}\beta\sigma_{x_1}^{-\varepsilon_1}$ and $\sigma_{x_1}^{\varepsilon_1}\sigma_{x_n}^{-\varepsilon_1}$ lie in $\dis{X}$ since $\dis{X}$ is a normal subgroup.
\end{proof}

Proposition \ref{prop:normal} implies that, for $x,y\in X$,
\[
[\sigma_x,\sigma_y]=\sigma_x^{-1}\sigma_y^{-1}\sigma_x\sigma_y\in \dis{X}.
\]
Therefore the commutator subgroup 
$\mathcal{G}(X)'=[\mathcal{G}(X),\mathcal{G}(X)]$ is a subgroup of $\dis{X}$. 
In the case of
connected racks, we have $\mathcal{G}(X)'= \dis{X}$.
This is not true in the case of solutions, as we shall see on the following example.
Moreover, the example also shows that, $\dis{X}$ and $\langle\sigma_e\rangle$ may have a nontrivial intersection.

\begin{example}\cite[Example 3.2]{JPZ21}
 Let $X=\Z_{p^2}$, for some prime~$p$, and let $\sigma_{i}(j)=j+pi+1$ and $\tau_j(i)=i-1-p(j+1)$, for all~$i,j\in X$. 
Then $(X,\sigma,\tau)$ is an indecomposable solution of multipermutation level~2 and $\mathcal{G}(X)\cong \Z_{p^2}$.
 It is easy to see that $\dis{X}=\{ j\mapsto j+pk \mid k\in\Z\}$ and therefore $\sigma_{i}^p\in\dis{X}$, for all~$i\in X$.
\end{example}

 \begin{remark}\label{rem:2m}
Let $n\in \mathbb{N}$, $(G,+,0)$ be an abelian group and $\alpha\in\aut{G}$ be of an order dividing~$n$. Let $(B,+,\circ)$ be a left non-trivial brace with a transitive cycle base $X$. 
Let $g\in X$ and $(B,\sigma,\tau)$ be a uniconnected solution with $\sigma_x(b):=(\lambda_x(g))^{-}\circ b$ and $\sigma_x^{-1}(b)=\lambda_x(g)\circ b$, for 
$x,b\in B$. Note that by \cite [Lemma 3.2]{C21}, in cyclic braces, $g$ is a generator of the additive group $(B,+)$.

It is well known (see e.g. \cite[Proposition 7]{Rump07}) that for a non-zero left brace $(B,+,\circ)$, the solution associated to the factor left brace $B/\Soc(B)$ is equal to the retraction of the solution associated to $(B,+,\circ)$. Further, by results of Castelli \cite[Theorem 3.5]{C21}, the multipermutation level of the solution associated with a brace $(B,+,\circ)$ coincides with the multipermutation level of the uniconnected solution associated with $(B,+,\circ)$. Moreover, by \cite[Proposition 6]{CGS17} the multipermutation level of the solution associated with $(B,+,\circ)$ is equal to $2$ if and only if $\widetilde{B}=B/\Soc(B)\neq \{0\}$ and $\widetilde{B}/\Soc(\widetilde{B})=\{0\}$. 
This gives that for each $x,y\in B$ there exists $s\in \Soc(B)$ such that
$x\circ y=x+y+s$
and in consequence for each $x\in X$ there exists $s_x\in \Soc(B)$ such that
$\sigma_x \sigma_0^{-1}(b)=
s_{x}\circ b$.
Since for $s_1,s_2\in \Soc(B)$, $s_1\circ s_2=s_1+s_2$, each $\gamma\in\dis{B}=\langle \sigma_x\sigma_{0}^{-1}\mid x\in B\rangle$ may be understood as the left translation (with respect to the operation $+$) by some $s\in \Soc(B)$.
Thus, $ \dis{B}$ is isomorphic to a subgroup of $(\Soc(B),+)$. 
If the left brace $(B,+,\circ)$ is cyclic, the group $\dis{B}$ is cyclic too.
\vskip 2mm

Now let $\gamma\in \dis{B}\cap\langle\sigma_0\rangle$. Then there is $s\in \Soc(B)$ such that $\gamma(b)=s\circ b$. On the other hand, there exist $m\in \mathbb{N}$ and $s'\in \Soc(B)$ with $\gamma(b)=\sigma_{0}^{m}(b)=(-m\cdot g+s')\circ b$. 
This implies that $m\cdot g\in \Soc(B)$ and $\sigma_{0}^{m}(b)=s\circ b$, for some $s\in \Soc(B)$.
This gives that the groups $\dis{X}$ and $\langle\sigma_0\rangle$ may have a nontrivial intersection.

Since by assumption $\Soc(B)\neq B$ and if the group $(B,+)$ is cyclic generated by $g$ then $g\notin \Soc(B)$. Hence, for example, groups $\dis{B}$ and $\langle\sigma_0\rangle$ have a trivial intersection if $\langle\sigma_0\rangle\cong \Z_2$.
 \end{remark}

 \begin{example}\label{exm:unisemi}
Let $(B,+,\circ)=(G\times \Z_n,+,\circ)$  be the $\alpha$-semidirect product of trivial braces constructed in Example \ref{exm:unisol}. 
Let~$X$ be a transitive cycle base of~$(B,+,\circ)$, 
 $(g,h)\in X$ and let $(G\times \Z_n, \sigma,\tau)$ with 
 \[
\sigma_{(a,i)}((b,j))=(\alpha^{-h}(b)-\alpha^{i-h}(g),j-h),
\]
for $(a,i),(b,j)\in G\times \Z_n$ be the uniconnected solution  described in Example \ref{exm:unisol}. 

Thus we have
$
\sigma_{(0,0)}((b,j))=(\alpha^{-h}(b)-\alpha^{-h}(g),j-h),
$
and
\[
\sigma^k_{(0,0)}((b,j))=(\alpha^{-kh}(b)-\sum_{r=1}^{k}\alpha^{-rh}(g),j-kh).
\]
Further, for $(a,i)\in G\times \Z_n$,
\[
\sigma_{(a,i)}\sigma^{-1}_{(0,0)}((b,j))=\sigma_{(a,i)}((\alpha^{h}(b)+g,j+h))=
(b+\alpha^{-h}(g)-\alpha^{i-h}(g),j)=
(b,j)+(\alpha^{-h}(g)-\alpha^{i-h}(g),0).
\]

Since $X$ is a transitive cycle base of~$(B,+,\circ)$ and $(g,h)\in X$, the set $\{\alpha^{i-h}(g)\mid i\in \Z_n\}$  generates the group $(G,+,0)$ and $\dis{G\times \Z_n}\cong (G,+,0)$. Moreover, $h$ is a generator of the group $\Z_n$. It is possible to prove that
$\sum_{i=1}^n \alpha^i(b)=0$, thus the order of $\sigma_{(0,0)}$ is equal to $n$ and it is easy to see that $\dis{G\times \Z_n}\cap\langle\sigma_{(0,0)}\rangle=\{\id\}$.
 \end{example}




\begin{example}\label{exm:dihdis}
Let $(\Z_{2^m},\sigma,\tau)$ be a uniconnected solution described in Example \ref{exm:Rump} with the permutation group isomorphic to the dihedral group $D_{2^m}$. In this case $\sigma_0(b)=1-b=\sigma_0^{-1}(b)$, 
and 
\[
\sigma_a\sigma_0^{-1}(b)=\begin{cases}b, \quad{\rm if}\; a\; {\rm is\; even,}\\
(2^m-1)(1-b)=b-2, \quad{\rm if}\; a\; {\rm is\; odd.}
\end{cases}
\]
Hence, $\dis{\Z_{2^m}} \cong(\Soc(\Z_{2^m}),+)=(2\Z_{2^m},+)\cong \Z_{2^{m-1}}$. Moreover $\langle\sigma_0\rangle\cong \Z_2$ and $\dis{\Z_{2^m}}\cap\langle\sigma_0\rangle=\{\id\}$. Consequently, $\mathcal{G}(\Z_{2^m})=\dis{\Z_{2^m}}\rtimes \langle\sigma_0\rangle\cong \Z_{2^{m-1}}\rtimes \Z_{2}$.

Similarly, for a uniconnected solution $(\Z_{2^m},+,\circ)$ with the permutation group isomorphic to the generalized quaternion group $Q_{2^m}$, we have $\dis{\Z_{2^m}}\cong(\Soc(\Z_{2^m}),+)=(2\Z_{2^m},+)$. 

But in this case, $\sigma_0(b)=(1-2^{m-1})(1-b)$, 
and 
\[
\sigma_a\sigma_0^{-1}(b)=\begin{cases}b, \quad{\rm if}\; a\; {\rm is\; even,}\\
(2^{m-1}-2)+b, \quad{\rm if}\; a\; {\rm is\; odd}
\end{cases}
\]
which gives
$\langle\sigma_0\rangle\cong \Z_4$ and $\dis{\Z_{2^m}}\cap\langle\sigma_0\rangle=\{\id,\sigma_0^2\}$. 

Note also that $\sigma_1^2=\sigma_0^2$, which implies $\sigma_1\sigma_0^{-2}=\sigma_1^{-1}$ and
\[
\mathcal{G}(\Z_{2^m})=\langle\alpha:=\sigma_1\sigma_0^{-1},\; \beta:=\sigma_0\mid \alpha^{2^{m-1}}=\id, \; \beta^2=\alpha^{2^{m-2}},\; \beta\alpha\beta^{-1}=\alpha^{-1}\rangle\cong Q_{2^m}.
\]
\end{example}

By \cite[Proposition 4.1]{JPZ21} the abelian  permutation group $\mathcal{G}(X)$ of an indecomposable multipermutation solution $(X,\sigma,\tau)$ of level $2$ is generated by at most two elements and all the permutations $\sigma_x$, for $x\in X$, have the same order.
Surprisingly, the same is also true for   any indecomposable multipermutation solution of level $2$.
\begin{proposition}\label{prop:3}
 Let $(X,\sigma,\tau)$ be an indecomposable
 solution of multipermutation level at most~$2$. Then
 \begin{enumerate}
  \item the permutation group $\mathcal{G}(X)$ is generated by at most two elements,
  \item for all $x\in X$, $\sigma_x$ have the same order,
  \item $\dis{X}=\langle L_{\widetilde i} \mid i\in \Z\rangle=\langle D_{\widetilde i} \mid i\in \Z\rangle$.
 \end{enumerate}
\end{proposition}

\begin{proof}
 Let $x\in X$.
 The permutation group $\mathcal{G}(X)$ is transitive and therefore
 there exist $k\in\N$ and $y_1,\ldots,y_k\in X$, such that
  \[x=\sigma_{y_1}^{\varepsilon_1}\sigma_{y_2}^{\varepsilon_2}\cdots \sigma_{y_k}^{\varepsilon_k}(\widetilde 0)=(L_{y_1}\pi)^{\varepsilon_1} (L_{y_2}\pi)^{\varepsilon_2}\cdots (L_{y_k}\pi)^{\varepsilon_k}(\widetilde 0), \; {\rm with}\; \varepsilon_i\in \{-1,1\}.\] 
Since the displacement group is a normal subgroup of $\mathcal{G}(X)$ there is $\alpha\in \dis{X}$,
such that
\[
(L_{y_1}\pi)^{\varepsilon_1} (L_{y_2}\pi)^{\varepsilon_2}\cdots (L_{y_k}\pi)^{\varepsilon_k}(\widetilde 0)=\alpha\pi^k(\widetilde 0).
\] Then, by \eqref{eq:disX}
 \[ \sigma_x=L_x\pi = L_{\alpha\pi^k(\widetilde0)}\pi = L_{\pi^k(\widetilde 0)}\pi=\sigma_{\widetilde k}\stackrel{\eqref{sigma_i}}{=}\sigma_{\widetilde 1}^k\sigma_{\widetilde 0}^{1-k}
 \]
 and therefore the permutation group $\mathcal{G}(X)$ is generated
 by $\sigma_{\widetilde 0}$ and $\sigma_{\widetilde 1}$,
 proving (1).
 
 By an easy inductive argument on $j$, we have, for all integers $i\leq j+1$,
 \begin{equation}
    \prod_{k=i}^j D_{\widetilde k} = L_{\widetilde j}L_{\widetilde{i-1}}^{-1}\label{L_to_D1}.
 \end{equation}
 In particular, for $i=1$ and $j\geq 0$ we obtain
 \begin{equation}
    L_{\widetilde j}=\prod_{k=1}^j D_{\widetilde k}\label{L_to_D2}.
 \end{equation}
Moreover, for $i\leq 0$ and $j=0$ we have
 \begin{equation}
    L_{\widetilde i}=\prod_{k=1+i}^{0} D_{\widetilde {k}}^{-1}.\label{L_to_D3}
 \end{equation}
 
 We have already seen that, for each $x\in X$, there exists~$k\in\Z$ such that
 $\sigma_x=\sigma_{\widetilde k}$ and therefore $L_x=L_{\widetilde k}$.
 Hence $\dis{X}=\langle L_x\mid x\in X\rangle=\langle L_{\widetilde k}\mid k\in \Z\rangle =\langle D_{\widetilde k}\mid k\in \Z\rangle$, proving~(3).

 Finally, let $\pi^r=\mathrm{id}$, for some $r\in\N$. By \eqref{piLxpi} we have for $i,k\in \Z$
 \[
 \pi L^{-1}_{\widetilde{k}}L_{\widetilde{i+k}}=L^{-1}_{\widetilde{k+1}}L_{\widetilde{i+k+1}}\pi.
 \] Hence we obtain 
 \[
  (\sigma_{\widetilde i})^r=(L_{\widetilde i}\,\pi)^r
  \stackrel{\eqref{piLxpi}}{=}
  \prod_{k=0}^{r-1} (L_{\widetilde k}^{-1} L_{\widetilde{i+k}})  \pi^r=
  \prod_{k=0}^{r-1} L_{\widetilde k}^{-1}  \prod_{k=0}^{r-1} L_{\widetilde{i+k}}=
  \mathrm{id}
 \]
 since $\widetilde m=\widetilde{m+r}$, for all~$m\in\Z$,
 and therefore $\{\widetilde 0,\ldots,\widetilde{r-1}\}=\{\widetilde{i},\ldots,\widetilde{i+r-1}\}$.
Hence the order of~$\sigma_{\widetilde i}$ divides the order of~$\sigma_{\widetilde 0}$.
 But $\widetilde 0$ was chosen randomly and thus we obtain~(2). 
 \end{proof}
  
Let us recall that a mapping $\Phi\colon X\to X'$ is a \emph{homomorphism} of two solutions $(X,\sigma,\tau)$ and $(X',\sigma',\tau')$ if, for each $x\in X$,
\[
\Phi\sigma_x=\sigma'_{\Phi(x)}\Phi.
\]
We shall then denote by $\Phi_{\mathcal{G}}$ the homomorphism from the group $\mathcal{G}(X)$ to the group $\mathcal{G}(X')$ defined by $\sigma_x\mapsto \sigma'_{\Phi(x)}$.
And by $\Phi_{D}$ we mean the restriction
of~$\Phi_{\mathcal{G}}$ to the homomorphism
$\dis{X}\to\dis{X'}$. 
Note that, for $\chi\in \mathcal{G}(X)$ and  $x\in X$,
$$\Phi(\chi(x))=\Phi_{\mathcal{G}}(\chi)\Phi(x).$$

\begin{corollary}\label{cor:semi-regular}
 Let $(X,\sigma,\tau)$ and $(Y,\rho,\upsilon)$ be two indecomposable solutions of multipermutation level~$2$. Let $\Phi_1:X\to Y$ and $\Phi_2:X\to Y$
 be two homomorphisms of the solutions. If there exists $e\in X$ such that $\Phi_1(e)=\Phi_2(e)$
 then $\Phi_1=\Phi_2$.
\end{corollary}

\begin{proof}
 According to Proposition~\ref{prop:3}, the permutation group is generated by~$\sigma_e$ and $\sigma_{\sigma_e(e)}$.
 Let now~$d\in X$. Since~$(X,\sigma,\tau)$ is indecomposable, there exist numbers~$k$ and $\varepsilon_1,\ldots,\varepsilon_k$ such that
 \[d= \sigma_e^{\varepsilon_1}\sigma_{\sigma_{e}(e)}^{\varepsilon_2}
 \sigma_{e}^{\varepsilon_3}
 \sigma_{\sigma_{e}(e)}^{\varepsilon_4}
 \cdots
 \sigma_{e}^{\varepsilon_k}(e).\]
 Now
 \begin{multline*}
 \Phi_1(d)=\Phi_1(\sigma_e^{\varepsilon_1}\sigma_{\sigma_{e}(e)}^{\varepsilon_2}
 \sigma_{e}^{\varepsilon_3}
 \sigma_{\sigma_{e}(e)}^{\varepsilon_4}
 \cdots
 \sigma_{e}^{\varepsilon_k}(e))
 =
 \rho_{\Phi_1(e)}^{\varepsilon_1}\rho_{\rho_{{\Phi_1(e)}}({\Phi_1(e)})}^{\varepsilon_2}
 \cdots
 \rho_{{\Phi_1(e)}}^{\varepsilon_k}({\Phi_1(e)})=\\
 \rho_{\Phi_2(e)}^{\varepsilon_1}\rho_{\rho_{{\Phi_2(e)}}({\Phi_2(e)})}^{\varepsilon_2}
 \cdots
 \rho_{{\Phi_2(e)}}^{\varepsilon_k}({\Phi_2(e)})
=
 \Phi_2(\sigma_e^{\varepsilon_1}\sigma_{\sigma_{e}(e)}^{\varepsilon_2}
 \sigma_{e}^{\varepsilon_3}
 \sigma_{\sigma_{e}(e)}^{\varepsilon_4}
 \cdots
 \sigma_{e}^{\varepsilon_k}(e))=\Phi_2(d). \qedhere
 \end{multline*}
\end{proof}

 \section{A construction of  indecomposable  solutions of multipermutation level 2 with non-abelian permutation group}\label{sec:constr}
  
  In \cite[Example 2.6]{JPZ21} we, together with Zamojska-Dzienio, provided an example of an indecomposable solution of multipermutation level $2$ of even order  with non-abelian permutation group. In this section we extend the idea applied there and we construct a wider class of solutions of multipermutation level $2$ with non-abelian permutation group.

Let us denote the group $(\Z,+,0)$ by $\Z_{\infty}$ which will be helpful for our further considerations. Additionally, in the case $n=\infty$ we assume that $a\equiv b\pmod n$ if and only if $a=b$.
\begin{theorem}\label{th:main}
Let $n\in \mathbb{N}\cup \{\infty\}$ and 
$(G,+,0)$ be an abelian group.
Further, let 
${\bf c}=(c_i)_{i\in\Z_n}\in G^n$ be a sequence 
such that $c_{0}=0$ and 
the group $(G,+,0)$ is generated by the set ~$\{c_{i}\mid i\in\Z_n\}$.

\noindent
Then $(G\times \Z_n,\sigma,\tau)$ with
 \[ \sigma_{(a,i)}((b,j))=(b+c_{i-j-1}-c_{-j-1},j+1)\]
  and
 \[\tau_{(a,i)}((b,j))=(b-c_{i-j+1}+c_{-j},j-1) \] 
is 
an indecomposable solution of multipermutation level $2$. 
Moreover, for any~$(e,k)\in G\times \Z_n$,
$$\mathcal{G}(G\times \Z_n)=\dis{G\times \Z_n}\rtimes \langle\sigma_{(e,k)}\rangle.$$

\end{theorem}

\begin{proof}

Let us first define a matrix of constants $c_{i,j}=c_{i-j}-c_{-j}$, for $i,j\in\Z_n$. 
Note that, for each  $j\in \Z_n$,  $c_{j,0}=c_j$, $c_{0,j}=0$  and $(G,+,0)= \langle c_{i,j}\mid i\in \Z_n\rangle$. With this
matrix we define,
by Theorem \ref{thm:2red}, a $2$-reductive solution
~$(G\times \Z_n,L,\mathbf{R})$ with
 \[L_{(a,i)}((b,j))=(b+c_{i,j},j)\quad {\rm and}\quad 
 \mathbf{R}_{(a,i)}((b,j))=(b-c_{i,j},j),
 \]
 for $(a,i), (b,j)\in  G\times \Z_n$. Note that $L_{(a,i)}=L_{(0,i)}$, for every $a\in G$ and $i\in \Z_n$,
 and that $L_{(0,0)}=\mathrm{id}$.
%
 
 Now, let $\pi:G\times \Z_n \to G\times \Z_n$ be $\pi((a,i))=(a,i+1)$. Clearly, this is a permutation on the set $G\times \Z_n$ with
 the  inverse mapping  $\pi^{-1}:G\times \Z_n \to G\times \Z_n$ given by $\pi^{-1}((a,i))=(a,i-1)$.


To complete the construction of a solution of multipermutation level $2$ we need to check  Condition \eqref{pi}. 
%
\noindent
 Let us compute, for $(a,i), (b,j), (e,k)\in  G\times \Z_n$,
 \begin{align*}
  L_{\pi((a,i))}\pi L_{(b,j)} ((e,k))&=
  L_{(a,i+1)} \pi
  ((e+c_{j,k},k))
  =
  L_{(a,i+1)}
  ((e+c_{j,k},k+1))
  \\&=
  (e+c_{j,k}+c_{i+1,k+1},k+1)\\
  &=(e+c_{j-k}-c_{-k}+c_{i-k}-c_{-k-1},k+1),
  \\
  L_{\pi((b,j))}\pi L_{(a,i)} ((e,k))&=
  L_{(b,j+1)} \pi
  ((e+c_{i,k},k))
  =
  L_{(b,j+1)}
  ((e+c_{i,k},k+1))
  \\&=
  (e+c_{i,k}+c_{j+1,k+1},k+1)\\
  &=(e+c_{i-k}-c_{-k}+c_{j-k}-c_{-k-1},k+1).
 \end{align*}

Then, by Theorem \ref{th:2per},
the solution $(G\times \Z_n,\sigma,\tau)$ with 
$$\sigma_{(a,i)}((b,j))=L_{(a,i)}\pi((b,j))=L_{(a,i)}((b,j+1))=(b+c_{i,j+1},j+1),$$ 
for $(a,i), (b,j)\in 
G\times \Z_n$,
is a 2-permutational solution. 

Recall, the solution is indecomposable, if the permutation group
$\mathcal{G}(X)$ is transitive. 
Let us have $a\in G$ arbitrary. We first  prove that
$(a +c_i,0)$ is in the same orbit as $(a,0)$, for any $i\in \Z_n$:
\[
\sigma_{(0,i)}\sigma_{(0,0)}^{-1}((a,0))
\stackrel{\eqref{dis}}{=}
L_{(0,i)}((a,0))=
(a+c_{i},0).
\]
By an induction argument then $(a,0)$ lies in the same orbit as $(0,0)$, for any $a\in G$. 
Now, for any $b\in G$ and $j\in\Z_n$,
\[ \sigma_{(0,0)}^j ((b,0))=\pi^j((b,0))=(b,j)\]
and therefore the permutation group $\mathcal{G}(X)$  is transitive on~$G\times \Z_n$.

Let $(a,i),(b,j)\in G\times \Z_n$. By the construction of permutations $L_{(a,i)}$ and $\sigma_{(b,j)}$, the intersection $\dis{G\times \Z_n}\cap\langle{\sigma_{(b,j)}}\rangle$ is trivial. Hence,  Proposition \ref{prop:normal} finishes the proof.
 \end{proof}

We will denote the solution described above by $\mathcal{S}(G\times \Z_n,{\bf c})$, for $n\in \mathbb{N}\cup\{\infty\}$, a group $(G,+,0)$ and ${\bf c}\in G^n$ 
specified in Theorem \ref{th:main}. 

Let us note that $\mathcal{S}(\Z_1\times \Z_n,{\bf 0})$ is a permutation solution. On the other hand,  $n=1$ forces $|G|=1$ and $\mathcal{S}(\Z_1\times \Z_1,0)$ is a trivial one element solution. Hence, from now on, we can assume $n>1$.

\begin{proposition}\label{prop:dis00}
Let $(X,\sigma,\tau)=\mathcal{S}(G\times \Z_n,{\bf c})$ and let $\Psi\colon \dis{X}\to G$  be defined by $\prod L_{(0,i)}^{p_i}\mapsto \sum p_i\cdot c_i$. Then the mapping
$\Psi$ is a well-defined epimorphism and
$\Ker\Psi=\dis{X}_{(0,0)}$.
In particular, $(G,+,0)\cong\dis{X}/\dis{X}_{(0,0)}$.
\end{proposition}
\begin{proof}
Recall, the abelian group  $\dis{X}$ is generated by the set $\{L_{(0,i)} \mid i\in\Z_n\}$ and the group $(G,+,0)$ is generated by the set ~$\{c_{i}\mid i\in\Z_n\}$. Since $L_{(0,i)}((a,0))=(a+c_i,0)$, for all $a\in G$ and $i\in\N$, the mapping $\Psi$ can be viewed as
the action of~$\dis{X}$ on $G\times\{0\}$ and therefore it is a homomorphism from $\dis{X}$ onto $G$.
Clearly $\Ker\Psi=\dis{X}_{(0,0)}$ and hence $(G,+,0)\cong\dis{X}/\dis{X}_{(0,0)}$.
\end{proof}






The following example presents in details the construction described in Theorem \ref{th:main}. 
\begin{example}\label{ex:4.3}
 Let $n=4$, $(G,+,0)=\Z_4\times \Z_2$.
 and let choose ${\bf c}=(c_0,c_1,c_2,c_3)\in (\Z_4\times \Z_2)^4$ as: $c_0=(0,0)$, $c_1=(1,0)$, $c_2=(1,0)$ and $c_3=(2,1)$. 
 Clearly, $\Z_4\times \Z_2=\langle c_i\colon i\in \Z_4\rangle$ and 
the $16$ constants $c_{i,j}$ are: 
$$\begin{array}{c|c|c|c|c}
c_{i,j}&0&1&2&3\\
\hline
1&(1,0)&(2,1)&(1,1)&(0,0)\\
2&(1,0)&(3,1)&(3,0)&(1,1)\\
3&(2,1)&(3,1)&(0,0)&(3,0)
\end{array}$$
Then $(\Z_4\times \Z_2\times\Z_4,L,\mathbf{R})$ with 
$$L_{(a_1,a_2,i)}((b_1,b_2,j))=((b_1,b_2)+c_{i,j},j)$$
is a $2$-reductive solution with an abelian permutation group generated by $\{L_{(0,0,i)} \mid i\in\{1,2,3\}\}$. 

It is easy to see that the group is isomorphic to the subgroup of the group~$G^4$ which
is generated by the vectors $\vec{c}_i=(c_{i,0},c_{i,1},c_{i,2},c_{i,3})$, for $i\in\{1,2,3\}$.
Moreover, we see, for each $i\in\{1,2,3\}$, that $c_{i,i}$ is not a linear combination
of the other two coefficients in the same column. Hence the permutation group has no less than three generators.
Furthermore, we have $2\vec{c}_1+2\vec{c}_2+2\vec{c}_3=\textbf{0}$ and therefore the group is not a free $\Z_4$-module.
A GAP (\cite{gap}) calculation showed us that the order of the group is~$2^5$.

Let us continue: the permutation $\pi\colon \Z_4\times \Z_2\times\Z_4\to \Z_4\times \Z_2\times\Z_4$ is given by $$\pi((a_1,a_2,i))= (a_1,a_2,i+ 1).$$

Finally, $(\Z_4\times \Z_2\times\Z_4,\sigma,\tau)$ with 
$$\sigma_{(a_1,a_2,i)}((b_1,b_2,j))=L_{(a_1,a_2,i)}\pi((b_1,b_2,j))=((b_1,b_2)+c_{i,j+1},j+1)$$
is an indecomposable solution $\mathcal{S}(\Z_4\times\Z_2\times \Z_4,{\bf c})$ of multipermutation level $2$ of size $2^5$.


Once again by GAP calculation, the group $\mathcal{G}(\Z_4\times \Z_2\times\Z_4)$ is of order~$2^7$ and is isomorphic to the group $(\Z_4^2\times \Z_2)\rtimes \Z_4$.
Since the order of~$\mathcal{G}(\Z_4\times \Z_2\times\Z_4)$ is greater than the size of~$\Z_4\times \Z_2\times\Z_4$, the solution
$\mathcal{S}(\Z_4\times\Z_2\times \Z_4,{\bf c})$ 
is not uniconnected.
\end{example}

In the sequel, 
let us have $(X,\sigma,\tau)=\mathcal{S}(G\times \Z_n,{\bf c})$, for some $n\in \mathbb{N}\cup\{\infty\}$, a group $(G,+,0)$ and ${\bf c}\in G^n$.



\begin{proposition}\label{prop:d}
The permutation group $\mathcal{G}(X)$ is abelian if and only if $c_{i\bmod n}=i\cdot c_1$, for all $i\in\Z$.
\end{proposition}

\begin{proof}
Taking in \eqref{piLx}, $\widetilde 0=(0,0)$, we directly obtain
\begin{align}\label{eq:Lpi}
\pi L_{(a,i)}\pi^{-1}= L_{\pi((0,0))}^{-1}L_{\pi((a,i))}=L_{(0,1)}^{-1}L_{(a,i+1)}
\end{align}
for $a\in G$ and $i\in\Z_n$.

Recall that the abelian group $\dis{X}$ is generated by the set $\{L_{(0,i)} \mid i\in\Z_n\}$. Then by \eqref{eq:Lpi}, the group  $\mathcal{G}(X)$ is abelian if and only if for every $(b,j)\in G\times \Z_n$ and all~$i\in\Z_n$
$$L_{(0,i)}((b,j))= L_{(0,i+1)}L_{(0,1)}^{-1}((b,j)).$$ 
This is equivalent
to, for all $i,j\in\Z_n$, 
\begin{align*}
c_{i+1,j}=c_{i,j}+c_{1,j}.
\end{align*}
This condition 
is equivalent to
$c_{i+1-j}=c_{i-j}+c_{1-j}-c_{-j}$,
 for all $i,j\in\Z_n$.
In particular, for $j=0$, we obtain $c_{i+1}=c_i+c_1$, which, by an easy induction, implies $c_{i\mod n}=i\cdot c_1$, for $i\in \Z$.
On the other hand, $c_{i\mod n}=i\cdot c_1$ implies $c_{i+1-j}+c_{-j}=c_{i-j}+c_{1-j}$ for $i,j\in \Z_n$.
\end{proof}
Note that by Proposition \ref{prop:d}, for a solution $\mathcal{S}(G\times \Z_n,{\bf c})$ with abelian permutation group, the group $(G,+,0)$ is cyclic. In a finite case, such solution with ${\bf c}=(c_i)_{i\in \Z_n}=(i)_{i\in \Z_n}$
overlaps with the solution $\mathcal{C}(m,n,0)$ described in \cite[Theorem 3.1]{JPZ21}.

\begin{proposition}\label{prop:lengthcycles}
For $x\in X$, the permutation $\sigma_x\in\mathcal{G}(X)$ is composed of cycles of length~$n$.
\end{proposition}

\begin{proof}
Let $(a,i),(b,j)\in G\times \Z_n$. By a straightforward induction we get for any $k\in \mathbb{N}$
\begin{align}\label{sigmal}
\sigma_{(a,i)}^k((b,j)) = ( b+\sum_{\ell=1}^{k} (c_{i-j-\ell}-c_{-j-\ell})\ ,\ j+k)
\end{align}
and therefore $\sigma_{(a,i)}^k((b,j))\neq (b,j)$ if~$k$ is not a multiple of~$n$. Now
\[
 \sum_{\ell=1}^{n} (c_{i-j-\ell}-c_{-j-\ell} )=\sum_{\ell\in\Z_n} c_\ell-\sum_{\ell\in\Z_n} c_\ell=0
\]
and therefore $\sigma_{(a,i)}^n((b,j)) = (b,j)$.
\end{proof}

\begin{proposition}\label{prop:trans}
 The automorphism group of the solution~$\mathcal{S}(G\times \Z_n,{\bf c})$ is regular.
\end{proposition}

\begin{proof}
  Let $g\in G$ and $k\in\Z_n$.
 For each $i\in\Z_n$ and $a\in G$, let
 \[\Phi((a,i))=
 \sigma_{(g,k)}^i ((a+g,k))=(a+\sum_{\ell=1}^i(c_{-\ell}-c_{-k-\ell})+g,k+i).
 \]
 This mapping is clearly a bijection
 and we shall now prove that it is a homomorphism of the solution.
 \begin{multline*}
  \Phi(\sigma_{(a,i)}((b,j))) =
  \Phi ((b+c_{i-j-1}-c_{-j-1},j+1))=\\
  (b+c_{i-j-1}-c_{-j-1}+\sum_{\ell=1}^{j+1}(c_{-\ell}-c_{-k-\ell})+g,k+j+1)=\\  
  (b+\sum_{\ell=1}^{j}(c_{-\ell}-c_{-k-\ell})+c_{i-j-1}-c_{-j-1}+c_{-j-1}-c_{-k-j-1}+g,k+j+1)\\  
  =(b+\sum_{\ell=1}^j(c_{-\ell}-c_{-k-\ell})+c_{i-j-1}-c_{-k-j-1}+g,k+j+1)=\\
  \sigma_{(\textrm{anything},k+i)} (b+\sum_{\ell=1}^j(c_{-\ell}-c_{-k-\ell})+g,k+j)= 
  \sigma_{\Phi((a,i))}
  \Phi((b,j)).
 \end{multline*}
Clearly $\Phi((0,0))=(g,k)$ and hence $\aut{X}$ is transitive and,
according to Corollary~\ref{cor:semi-regular}, it is regular.
\end{proof}

It is a question, which choice of parameters in the construction
yields the same solution (up to isomorphism). We show that the obvious
choice is the only possibility.

\begin{theorem}\label{thm:isoS}
Two solutions  $(X,\sigma,\tau)=\mathcal{S}(G\times \Z_n,{\bf c})$ and  $(X',\sigma',\tau')=\mathcal{S}(G'\times \Z_{n'},{\bf c}')$ are isomorphic if and only if 
\begin{enumerate}
\item $n'=n$
\item there exists a group isomorphism $g:G\to G'$ such that, for each $i\in\Z_n$
$$c_i'=g(c_i).$$ 
\end{enumerate}
\end{theorem}

\begin{proof}
  ``$\Rightarrow$'' Let $\Phi\colon G\times \Z_n\to G'\times \Z_{n'}$ be an isomorphism of the solutions. According to Proposition~\ref{prop:lengthcycles}, we have $n=n'$. According to~Proposition \ref{prop:trans}, we may suppose
  $\Phi((0,0))=(0,0)$. Since $\sigma_{(0,0)}((0,i))=(0,i+1)$, we have $\Phi((0,i))=(0,i)$, for each~$i\in\Z_n$.
  
By Proposition~\ref{prop:dis00}, groups $\dis{X}/\dis{X}_{(0,0)}$ and $(G,+,0)$ and $\dis{X'}/\dis{X'}_{(0,0)}$ and $(G',+,0)$ are isomorphic through the isomorphism $\Psi:\dis{X}/\dis{X}_{(0,0)}\to G$ such that $L_{(0,i)}\dis{X}_{(0,0)}\mapsto c_i$  and the isomorphism $\Psi':\dis{X'}/\dis{X'}_{(0,0)}\to G'$ such that $L'_{(0,i)}\dis{X'}_{(0,0)}\mapsto c'_i$. Let $g=\Psi'\Phi_{D}\Psi^{-1}$ and $\alpha\in \dis{X}$. Hence $\Phi_{D}(\alpha)(\Phi((0,0)))=\Phi(\alpha((0,0)))$. 
Since $\Phi((0,0))=(0,0)$ this implies that $\alpha \in \dis{X}_{(0,0)}$ if and only if $\Phi_D(\alpha)\in \dis{X'}_{(0,0)}$. Then
  \[g(c_i)=\Psi'\Phi_{D}(L_{(0,i)}\dis{X}_{(0,0)})=\Psi'(L_{(0,i)}'\dis{X'}_{(0,0)})=c_i'.\]

 

  ``$\Leftarrow$''  It is evident by straightforward calculations.
\end{proof}

\begin{example}
By Theorem \ref{thm:isoS}, solutions $\mathcal{S}(\Z_m\times \Z_2,{\bf c}=(0,1))$ and $\mathcal{S}(\Z_m\times \Z_2,{\bf c}=(0,g))$, where $g\in \Z^*_m$, 
are isomorphic. Hence there is exactly one 
non-isomorphic solution of the form  $\mathcal{S}(\Z_m\times \Z_2,{\bf c})$.
On the other hand, for any prime number $p$, there are $p+1$ non-isomorphic solutions of the form $\mathcal{S}(\Z_p\times \Z_3,{\bf c})$: $\mathcal{S}(\Z_p\times \Z_3,{\bf c}=(0,0,1))$
and $\mathcal{S}(\Z_p\times \Z_3,{\bf c}=(0,1,g))$,
for any $g\in\Z_p$.
\end{example}

 
We know that the displacement subgroup $\dis{G\times \Z_n}$ of a solution $\mathcal{S}(G\times \Z_n,{\bf c})$ is abelian but, by Proposition \ref{prop:dis00}, in general it is larger than the group~$(G,+,0)$. We shall show an example where this does not happen.


\begin{proposition}\label{prop:module}
Let $r,k\in \mathbb{N}$ and let~$(G,+,0)$ be a free $\Z_k$-module of rank~$r$.
 Let $e_1,\ldots,e_{r}$ be a free basis of~$(G,+,0)$.
 We set 
 $n=2r$, $c_0=0$, $c_i=\sum_{k=1}^i e_k$
 and $c_{i+r}=\sum_{k=i+1}^r e_k$,  for $i\in\{1,\ldots,r\}$. Then the solution $(X,\sigma,\tau)=\mathcal{S}(G\times \Z_n,{\bf c})$ is uniconnected, $(G,+,0)\cong\dis{X}$ and ${\mathcal{G}}(X)\cong G\rtimes_\alpha\Z_n$, where $\alpha(1)(e_i)=e_{i+1}$, for $i<r$, and $\alpha(1)(e_r)=-e_1$.
 
\end{proposition}

\begin{proof}
 The group generated by $\{L_{(0,i)} \mid i\in\Z_n\}$
 is clearly a~$\Z_k$-module. Actually, this module is generated by $\{L_{(0,1)},\ldots,L_{(0,r)}\}$ since $L_{(0,0)}$ is
 the identity permutation and $L_{(0,r+i)}=L_{(0,i)}^{-1}L_{(0,r)}$ for each $0< i<r$.

By Proposition \ref{prop:dis00}, $(G,+,0)\cong\dis{X}/\dis{X}_{(0,0)}$ and the stabilizer $\dis{X}_{(0,0)}$ is the kernel of the group epimorphism $\Psi\colon \dis{X}\to G$ defined by $L_{(0,i)}\mapsto c_i$. Then, since for each $\gamma\in \dis{X}$ there is $a\in G$ such that $\gamma((0,0))=(a,0)$ and $e_1,\ldots,e_{r}$ is a free basis of~$(G,+,0)$,
the kernel of~$\Psi$  is trivial and $(G,+,0)\cong\dis{X}$.

Clearly, for $(b,j)\in G\times \Z_n$, $\pi((b,j))=\sigma_{(0,0)}((b,j))=(b,j+1)$, hence the order of $\pi$ is equal to $n$ and by Theorem \ref{th:main}, ${\mathcal{G}}(X)\cong G\rtimes \Z_n$.
Moreover, $\Psi(D_{(0,i)})=\Psi(L_{(0,i)}L_{(0,i-1)}^{-1})=c_i-c_{i-1}=e_i$,
for $0<i\leq r$, and $\Psi(D_{(0,i)})=-e_{i-r}$, for $r<i\leq 2r$.
The rest follows from Equation~\eqref{piDx}.
\end{proof}

\begin{corollary}
 Each finite abelian group embeds into the permutation group of an indecomposable uniconnected solution
 of multipermutation level~$2$.
\end{corollary}

\begin{proof}
 Let~$(G,+,0)$ be a finite abelian group of an exponent~$k$ with~$r$ generators. Then $(G,+,0)$ embeds into the free $\Z_k$-module of rank~$r$. According to the Proposition \ref{prop:module},
 this module embeds into the permutation group of a uniconnected solution.
\end{proof}

\section{Homomorphic images}\label{sec:homo}

In this section we focus on homomorphic images of the solutions
constructed in the previous section. It turns out that there exists one universal indecomposable multipermutation solution of level~2, meaning that
every indecomposable multipermutation solution of level~2 is a homomorphic image of the universal one.
Since each image of an indecomposable multipermutation solution of level~2
is an indecomposable multipermutation solution of level at most~2,
we may then conclude that indecomposable multipermutation solutions of level at most~2
form the class of all the images of one universal solution.
This solution is constructed 
using the free abelian group of rank~$\omega$. 

Recall from Section~3 that,
for~$(X,\sigma,\tau)$ a solution, we choose~$\widetilde 0\in X$.  Let~$\pi=\sigma_{\widetilde 0}$ and let $\widetilde i=\pi^i(\widetilde 0)$, for each $i\in\N$.
 Let $L_{\widetilde i}=\sigma_{\widetilde i}\pi^{-1}$
 and $D_{\widetilde i}=L_{\widetilde i}L_{\widetilde{i-1}}^{-1}$.

\begin{proposition}\label{prop:homim}
Let $\bigoplus_\Z \Z$ be the free abelian group of rank~$\omega$ and let $\{e_i\mid i\in\Z\}$ be a free basis of $\bigoplus_\Z \Z$. Each indecomposable multipermutation solution of level at most~$2$ is a homomorphic image of the solution $\mathcal{S}((\bigoplus_\Z\Z)\times \Z,\mathbf{c})$, where
 \[c_i=\begin{cases}
        \sum_{k=1}^i e_k&\text{for }i >0,\\
        0&\text{for }i=0,\\
        \sum_{k=1}^{-i} -e_{1-k}&\text{for }i< 0.
       \end{cases}\]
\end{proposition}

\begin{proof}
Let~$(X,\sigma,\tau)$ be an indecomposable multipermutation solution of level at most~$2$.
 It is easy to see that $c_i-c_j=\sum_{k=j+1}^i e_k$, for any $i\geq j$. 
Let 
 us define $\Phi\colon (\bigoplus_\Z\Z)\times \Z\to X$ by 
 \[\Phi ((\sum r_k\cdot e_k,j))=\prod D_{\widetilde {k+j}}^{r_{k}} (\widetilde j).\]
 Now, for $i\geq 0$,
 \begin{align*}
  \Phi (\sigma_{(\sum r_k\cdot e_k,i)}((\sum s_k\cdot e_k,j)))&=\Phi((\sum s_k\cdot e_k+c_{i-j-1}-c_{-j-1},j+1))=\\
  \Phi((\sum s_k\cdot e_{k}+\sum_{\ell=-j}^{i-j-1}e_\ell,j+1))&=\prod D_{\widetilde{k+j+1}}^{s_k} \prod_{\ell=-j}^{i-j-1} D_{\widetilde{\ell+j+1}}(\widetilde{j+1})=\\
  &\prod D_{\widetilde{k+j+1}}^{s_k} \prod_{\ell=1}^{i} D_{\widetilde{\ell}}(\widetilde{j+1})
  \stackrel{\eqref{L_to_D2}}{=}
  L_{\widetilde i} \prod D_{\widetilde{k+j+1}}^{s_k}(\widetilde {j+1}),\\
 \sigma_{\Phi((\sum r_k\cdot e_k,i))} (\Phi((\sum s_k\cdot e_k,j))) &=
  \sigma_{\prod D_{\widetilde {k+i}}^{r_{k}}(\widetilde i)} \prod D_{\widetilde {k+j}}^{s_{k}} (\widetilde j)
  \stackrel{\eqref{eq:disX}}{=}
  \sigma_{\widetilde i} \prod D_{\widetilde {k+j}}^{s_{k}} (\widetilde j)=\\
  &L_{\widetilde i}\pi \prod D_{\widetilde {k+j}}^{s_{k}} (\widetilde j)
  \stackrel{\eqref{piDx}}{=}
  L_{\widetilde i}\prod D_{\widetilde {k+j+1}}^{s_{k}} \pi(\widetilde j)= L_{\widetilde i}\prod D_{\widetilde {k+j+1}}^{s_{k}} (\widetilde{j+1}).
 \end{align*}
  For $i< 0$ we use $c_{i-j-1}-c_{-j-1}=-\sum_{k=i-j}^{-j-1}e_k
  =\sum_{k=i+1}^{0} -e_k$ and then instead of $\eqref{L_to_D2}$ we apply
  $\eqref{L_to_D3}$. Hence $\Phi$ is a homomorphism which is clearly onto.
\end{proof}

The researchers tend to focus on finite solutions and therefore it is useful to know that a finite solution is an image of a finite construction.

\begin{proposition}\label{prop:homfin}
 Let~$(X,\sigma,\tau)$ be an indecomposable multipermutation solution of level at most~$2$. Then it is a homomorphic image of the solution $(Y,\rho,\upsilon)=\mathcal{S}(\dis{X}\times \Z_n,\mathbf{c})$, where $n=o(\pi)$ and $c_i=L_{\widetilde i}$.
\end{proposition}

\begin{proof}
We prove that the mapping $\Phi\colon Y\to X$ given by
 $(\gamma,j)\mapsto \pi^j\gamma(\widetilde {0})$ is a homomorphism of solutions: ~$(Y,\rho,\upsilon)$ and ~$(X,\sigma,\tau)$. Indeed, for $(\beta,i),(\gamma,j)\in \dis{X}\times \Z_n$
 \begin{multline*}
  \Phi(\rho_{(\beta,i)}((\gamma,j)))=\Phi
  ((\gamma L_{\widetilde{i-j-1}}L_{\widetilde{-j-1}}^{-1},j+1))=
  \pi^{j+1} L_{\widetilde{i-j-1}}L_{\widetilde{-j-1}}^{-1}\gamma(\widetilde {0})
 \stackrel{\eqref{piLxpi}}{=}\\
L_{\widetilde{i}}\pi^{j+1}L_{\pi^{-j-1}(\widetilde{i})}^{-1}L_{\widetilde{i-j-1}}\gamma(\widetilde 0)= 
L_{\widetilde{i}}\pi\pi^{j}\gamma(\widetilde 0)
=\sigma_{\widetilde{i}}\pi^j\gamma(\widetilde 0)
\stackrel{\eqref{eq:disX}}{=}
  \sigma_{\pi^i\beta(\widetilde 0)}\pi^j\gamma(\widetilde 0)=
  \sigma_{\Phi((\beta,i))}(\Phi((\gamma,j))).
 \end{multline*}
The homomorphism~$\Phi$ is onto since the permutation group $\mathcal{G}(X)$
is transitive.
\end{proof}

We can see in the proof that we can replace the group~$\dis{X}$ by the group $\dis{X}/\dis{X}_{\tilde 0}$ and it works as well.
This gives us a characterization of all the solutions that can be obtained by the construction described in Theorem  \ref{th:main}.
\begin{corollary}
An indecomposable multipermutation solution ~$(X,\sigma,\tau)$ of level at most~$2$ is isomorphic to the solution of the form $\mathcal{S}(G\times \Z_n,\mathbf{c})$, for some $n\in \mathbb{N}$, a group $(G,+,0)$ and a sequence   $\mathbf{c}\in G^n$, if and only if $\dis{X}\cap \langle \pi\rangle=\{\id\}$ and $\mathcal{G}(X)_{\widetilde {0}}\subseteq\dis{X}$.
\end{corollary}
\begin{proof}
``$\Leftarrow$'' 
By Proposition \ref{prop:homfin}  it is sufficient to show that for solutions~$(X,\sigma,\tau)$ with $\dis{X}\cap \langle \pi\rangle=\{\id\}$ and $\mathcal{G}(X)_{\widetilde {0}}\subseteq\dis{X}$, the mapping  $\Phi\colon X\to \dis{X}/{\dis{X}_{\widetilde {0}}}$ given by
 $(\gamma,j)\mapsto \pi^j\gamma(\widetilde {0})/{\dis{X}_{\widetilde {0}}}$ is an injection. Let $(\gamma,j),(\gamma',j')\in \dis{X}\times \Z_n$. Then 
 \[
\pi^j\gamma(\widetilde {0})= \pi^{j'}\gamma'(\widetilde {0})\;\Leftrightarrow\;(\gamma')^{-1}\pi^{j-j'}\gamma(\widetilde {0})= \widetilde {0}\;\Leftrightarrow\;(\gamma')^{-1}\pi^{j-j'}\gamma\in \mathcal{G}(X)_{\widetilde {0}}.
 \]
The assumption $\mathcal{G}(X)_{\widetilde {0}}\subseteq\dis{X}$ forces $\pi^{j-j'}$ to belong to $\dis{X}$. Further, by  $\dis{X}\cap \langle \pi\rangle=\{\id\}$, we obtain $\pi^{j-j'}=\id$. Hence, $(\gamma')^{-1}\gamma\in \dis{X}_{\widetilde {0}}$, which proves that $\gamma\equiv \gamma'\pmod {\dis{X}_{\widetilde {0}}}$.
 
``$\Rightarrow$'' It follows from the construction of the solution $\mathcal{S}(G\times \Z_n,\mathbf{c})$. 
\end{proof}

In particular, whenever
$(X,\sigma,\tau)$ is uniconnected and $\dis{X}\cap \langle \pi\rangle$ is trivial, the mapping $\Phi$ from the proof of Proposition \ref{prop:homfin} is an isomorphism.
One such example comes from Example \ref{exm:unisemi}.
\begin{remark}\label{rem:unisemi}
Let us consider the solution $(G\times \Z_n, \sigma,\tau)$  from Example \ref{exm:unisemi}. In this case $\dis{G\times \Z_n}\cong (G,+,0)$, 
\[
\pi((b,j))=\sigma_{(0,0)}((b,j))=(\alpha^{-h}(b)-\alpha^{-h}(g),j-h),
\]
and for $i\in \Z_n$
\begin{multline*}
L_{\widetilde i}((b,j))=L_{\pi^i((0,0))}((b,j))=\sigma_{\pi^i((0,0))}\sigma^{-1}_{(0,0)}((b,j))=\sigma_{(-\sum_{r=1}^{i}\alpha^{-rh}(g),-ih)}\sigma^{-1}_{(0,0)}((b,j))=\\
(b+\alpha^{-h}(g)-\alpha^{-ih-h}(g),j).
\end{multline*}
Therefore, we obtain that 
$(G\times \Z_n, \sigma,\tau)$
is isomorphic to the solution $\mathcal{S}(G\times \Z_n,{\bf c})$, for ${\bf c}=(c_i)_{i\in \Z_n}=(\alpha^{-h}(g)-\alpha^{-(i+1)h}(g))_{i\in \Z_n}\in G^n$
. Clearly, by Theorem \ref{thm:isoS} this solution is isomorphic to the solution $\mathcal{S}(G\times \Z_n,{\bf c})$, for ${\bf c}=(-\alpha^{-(i+1)h}(g))_{i\in \Z_n}\in G^n$. 
 \end{remark}


\begin{example}
By Remark \ref{rem:unisemi} the solution $(\Z_2\times \Z_2\times \Z_3,\sigma,\tau)$, with
 \[
\sigma_{((a_1,a_2),i)}(((b_1,b_2),j))=(\alpha^{2}((b_1,b_2))+\alpha^{i-1}((1,0)),j-1)=((b_1+b_2,b_1)+\alpha^{i-1}((1,0)),j-1),
\] and $\alpha=\left(\begin{smallmatrix}
0&1\\
1&1\end{smallmatrix}\right)$,
described in Example \ref{exm:uni223}, is isomorphic to the solution:
$$\mathcal{S}(\Z_2\times \Z_2\times \Z_3,{\bf c}=((0,0),(1,0),(0,1))).$$
In this case $c_{0,j}=(0,0)$, $c_{1,j}=(1,0)$ and $c_{2,j}=(0,1)$ and
 \[
\sigma_{((a_1,a_2),i)}(((b_1,b_2),j))=((b_1,b_2)+c_{i,j+1},j+1).
\]
\end{example}

\subsection*{Congruences}

In the sequel we focus on homomorphic images
of the solutions. It is well known that homomorphisms of solutions and equivalence relations preserving the structure of solutions are closely related:
let $(X,\sigma,\tau)$ be a solution. An equivalence relation $\mathord{\asymp}\subseteq X\times X$ such that for $x_1,x_2,y_1,y_2\in X$ 
\begin{align}\label{congr}
&x_1\asymp x_2\;\; {\rm and} \;\; y_1\asymp y_2\quad \Rightarrow\quad \sigma^{\varepsilon}_{x_1}(y_1)\asymp \sigma^{\varepsilon}_{x_2}(y_2),
\end{align}
where $\varepsilon\in \{-1,1\}$, is called a \emph{congruence} of the solution $(X,\sigma,\tau)$.

If $\Phi\colon X\to X'$ is a homomorphism
from a~solution~$(X,\sigma,\tau)$ to a solution~$(X',\sigma',\tau')$ then the \emph{kernel} of $\Phi$, defined by 
\begin{align*}
&x_1\kker\Phi~x_2\quad \Leftrightarrow\quad \Phi(x_1)=\Phi(x_2)
\end{align*}
is a congruence of~$(X,\sigma,\tau)$. 
On the other hand, for a congruence $\asymp$ of the solution $(X,\sigma,\tau)$, let $\chi\colon X\to X/\mathord{\asymp}$ defined by $(a,i)\mapsto (a,i)/\mathord{\asymp}$ be the \emph{natural projection} from the solution to the quotient one. Clearly, $\chi$ is onto and $ker\chi=\mathord{\asymp}$.

Moreover, for any epimorphism $\Phi\colon X\to X'$, the solution~$(X',\sigma',\tau')$ is isomorphic to the quotient solution $(X,\sigma,\tau)/{\kker\Phi}$. Hence, a solution is a homomorphic image of a solution $(X,\sigma,\tau)$ if and only if it is isomorphic to a quotient of  $(X,\sigma,\tau)$ by some congruence. Thus the problem of finding all homomorphic images of  $(X,\sigma,\tau)$ reduces to the problem of finding all congruences of $(X,\sigma,\tau)$.
\begin{lemma}\label{lem:5.3}
Let $(X,\sigma,\tau)=\mathcal{S}(G\times \Z_n,\mathbf{c})$ and $\asymp$ be a congruence of the solution $(X,\sigma,\tau)$.
 Then
 \begin{enumerate}
  \item 
 there exists~$H$, a subgroup of~$(G,+,0)$ such that, for all~$a,b\in G$ and each $i\in\Z_n$,
  \[ 
 (a,i)\asymp (b,i)\qquad \text{ if and only if}\qquad a-b\in H;
 \]
 \item there exists $m\in\N\cup\{\infty\}$ such that
 \begin{itemize}
 \item $m$ is a divisor of~$n$, if $n$ is finite,
  \item if 
   $(a,i)\asymp (b,j)$, for some $a,b\in G$ and $i,j\in\Z_n$, then $i-j\equiv 0\pmod m$,
  \item if $m$ is finite then $c_i-c_{i+m}\in H$, for all $i\in\Z_n$.
 \end{itemize}
 \end{enumerate}
\end{lemma}

\begin{proof}
Recall that, for $(a,i),(b,j)\in G\times \Z_n$, $L_{(a,i)}((b,j))=(b+c_{i-j}-c_{-j},j)$, $\pi((b,j))=(b,j+1)$ and $\sigma_{(a,i)}((b,j))=(b+c_{i-j-1}-c_{-j-1},j+1)$. 

(1) By definition, the group $(G,+,0)$ is generated by the set $\{c_{i}\colon i\in \Z_n\}$. It simply means that, for each $g\in G$ and $i\in\Z_n$,
there exists a permutation $\alpha\in\dis{G\times \Z_n}$ such that, for every $a\in G$,
\begin{align}\label{alpha} 
(a+g,i)=\alpha((a,i)).
\end{align} 
Hence, for each $a,b\in G$, $i\in\Z_n$ and $g\in G$,
\[
(a,i)\asymp(b,i)\quad \Rightarrow \quad\alpha((a,i))\asymp\alpha((b,i))\quad \Rightarrow \quad (a+g,i)\asymp(b+g,i).
\]
Analogously, using $\pi^k$ we prove, for $a,b\in G$ and $i,j,k\in \Z_n$,
\begin{equation}\label{i+k}
(a,i)\asymp(b,j)\quad \Rightarrow \quad (a,i+k)\asymp(b,j+k).
\end{equation}
Let $(a,i),(b,i),(a',i),(b',i)\in G\times \Z_n$ be such that $
(a,i)\asymp(b,i)\;\; {\rm and}\;\;(a',i)\asymp(b',i)$. 
Hence
\[
(a+a',i)\asymp(b+a',i)\asymp(b+b',i).
\]
Therefore, for each $i\in \Z_n$, 
$H_i=\{a\in G\mid (a,i)\asymp (0,i)\}$ is a subgroup of~$(G,+,0)$ and 
\begin{multline*}
a,a'\in H_i \quad \Leftrightarrow\quad (a,i)\asymp (a',i)\quad \Leftrightarrow\quad (0,i)\asymp (a-a',i)
\\ \Leftrightarrow\quad (0,0)\asymp (a-a',0)\quad \Leftrightarrow\quad a-a'\in H_0.
\end{multline*} 
In particular, for each $i\in \Z_n$, $H_i=H_0$.



(2) If $(a,i)\asymp(a',i')$ always implies $i=i'$ then we can take $m=n$ and the proof is finished. Suppose hence, for the rest of the proof, that there
exist $(a,i),(a',i')\in G\times\Z_n$ such that $(a,i)
\asymp(a',i')$ and $i\neq i'$.

Let $m$ be the smallest positive integer such that there exists $d\in G$ with
$(0,0)\asymp(d,m)$. 
Let $\gamma\in \dis{X}$ be such that $\gamma((0,m))=(d,m)$. 
Then $(0,0)\asymp \gamma\pi^m((0,0))$ and, by induction, $(0,0)\asymp
(\gamma\pi^m)^k((0,0))$, for any~$k\in\Z$.
If~$n$ is finite then we can set $k=\lceil\frac nm\rceil$ and we see that there exists some $b\in G$ such that $(0,0)\asymp (b,km)=(b,km-n)$. Since $0\leq km-n<m$ and~$m$ is minimal, necessarily $m$ divides~$n$.

Let again $(a,i)\asymp(a',i')$ for some $a,a'\in G$ and $i,i'\in \Z_n$. 
Let~$\beta\in\dis{X}$ be such that $\beta((0,0))=(a,0)$.
Then by \eqref{i+k}, 
\[(a,i)\asymp(a',i') \quad \Leftrightarrow \quad (a,0)\asymp(a',i'-i)
\quad \Leftrightarrow \quad (0,0)\asymp\beta^{-1}((a',i'-i)).
\]
Hence $i'\equiv i\pmod m$.

Finally, for each~$i\in\Z_n$,
\begin{multline}\label{cc}
(0,0)\asymp(d,m)
\quad \Leftrightarrow \quad (0,i)\asymp (d,m+i)
\quad \Rightarrow \quad L_{(0,i)}((0,0))\asymp L_{(d,m+i)}((0,0))\\
\quad \Leftrightarrow \quad (c_{i},0) \asymp (c_{m+i},0)
\quad \Leftrightarrow \quad c_i- c_{m+i}\in {H_0}.\qedhere
\end{multline}
\end{proof}

\begin{lemma}\label{lem:commutator}
 Let $(X,\sigma,\tau)$, $\asymp$, $H$ and~$m$ be as in Lemma~\ref{lem:5.3}. 
Let~$m$ be finite and
 let $\Phi\colon X\to X/\mathord{\asymp}$ be the natural projection.
 Then $[\pi^m,\dis{X}]\subseteq \Ker\Phi_{D}$.
\end{lemma}

\begin{proof}
For all $b\in G$ and $i,j \in\Z_n$ we have
\begin{multline*}
\pi^{-m}L_{(0,i)}^{-1}\pi^mL_{(0,i)}((b,j))=
\pi^{-m}L_{(0,i)}^{-1}\pi^{m}((b+c_{i-j}-c_{-j},j))
=\\
\pi^{-m}L_{(0,i)}^{-1}((b+c_{i-j}-c_{-j},j+m))
 =
 \pi^{-m}((b+c_{i-j}-c_{-j}-(c_{i-j-m}-c_{-j-m}),j+m))=\\
 (b+c_{i-j}-c_{i-j-m}+c_{-j-m}-c_{-j},j)=(b+h,j)
\end{multline*}
for some $h\in H$, according to Lemma~\ref{lem:5.3} (2). Hence,
by Lemma~\ref{lem:5.3} (1),
\[(b,j)\asymp(b+h,j)=[\pi^m,L_{(0,i)}]((b,j)).
\]
This gives 
$\Phi((b,j))=\Phi_{\mathcal{G}}([\pi^m,L_{(0,i)}]) \Phi((b,j))$ 
and therefore $[\pi^m,L_{(0,i)}]\in \Ker(\Phi_{\mathcal{G}})$.
Since $\dis{X}$ is a normal subgroup of~$\mathcal{G}(X)$, we obtain
$[\pi^m,L_{(0,i)}]\in\dis{X}$. Because $\dis{X}=\langle L_{(0,i)}\mid i\in \Z_n\rangle$, this completes the proof.
\end{proof}


\begin{proposition}\label{prop:image}
 An equivalence relation $\asymp$ is a congruence of the solution $(X,\sigma,\tau)=\mathcal{S}(G\times \Z_n,\mathbf{c})$ if and only if there exist a~subgroup~$H$ of~$(G+,0)$,  $m\in \mathbb{N}\cup\{\infty\}$ and $r\in G$ such that:
 \begin{enumerate}
  \item [i)] $m$ divides~$n$, if $n$ is finite;
  \item [ii)] if $m$ is finite then $c_i-c_{i+m}\in H$, for each~$i\in \Z_n$;
  \item [iii)] if $n$ and $m$ are finite then $\frac nm\cdot r\in H$;
  \item [iv)] for all $a,a'\in G$ and $i,i'\in\Z_n$
  \begin{multline}\label{def:con}
  (a,i)\asymp(a',i')\quad \Leftrightarrow\quad i-i'\equiv 0\pmod m\ \text{and}\ 
  \begin{cases}
  a'-a\equiv\frac{i-i'}m\cdot r\pmod H &\text{if }
  m<\infty\\
  a'-a\in H & \text{if }m=\infty.
  \end{cases}
  \end{multline}
 \end{enumerate}
\end{proposition}

\begin{proof}
Let $\asymp$ be a congruence of the solution $(X,\sigma,\tau)$.
The items (i) and (ii) were proved in~Lemma~\ref{lem:5.3}.
If~$n=m$ then (iii) and (iv) also follow from Lemma~\ref{lem:5.3}.
Suppose hence~$m<n$.

Recall that, if $n$ is finite, $m$ is the smallest positive integer such that there exists $d\in G$ with
$(0,0)\asymp(d,m)$. Let now~$\gamma\in\dis{X}$
be such that $(0,0)\asymp\pi^m\gamma(0,0)$ and denote by~$r$ the element from~$G$ such that $\gamma((0,0))=(r,0)$.
Let $\iota=[\gamma,\pi^m]$. By Lemma~\ref{lem:commutator},  $\iota((0,0))\asymp (0,0)$. By Lemma~\ref{lem:5.3}, there exists $h\in H$ such that
$\iota((0,0))=(h,0)$.

Then
\[\textstyle (\pi^m\gamma)^{n/m}((0,0))=\pi^n\gamma^{n/m}\iota^{n/m \choose 2}((0,0))
 =(\frac nm\cdot r+{n/m \choose 2}\cdot h,0).
\]
Since $\pi^m\gamma((0,0))\asymp(0,0)$, then by~Lemma~\ref{lem:5.3}
we obtain $r\cdot \frac nm\in H$.

Further, we will show that for $a,b\in G$ and $k\in \Z_n$ with $m\mid k$, if $(a,0)\asymp(b,k)$ then 
$(0,0)\asymp(b-a,k)$. Let $\chi\in \dis{X}$ be such that $\chi((a,0))=(0,0)$. Assume $\chi=L_{(0,i_1)}^{p_1}\ldots L_{(0,i_s)}^{p_s}$ for some numbers $i_1,\ldots,i_s\in \Z_n$ and $p_1,\ldots,p_s\in \Z$. This implies 
\[
\chi((a,0))=L_{(0,i_1)}^{p_1}\ldots L_{(0,i_s)}^{p_s}((a,0))=(a+\sum_{\ell=1}^{s} p_\ell\cdot c_{i_\ell},0)=(0,0)
\]
and therefore $\sum p_{\ell}\cdot c_{i_{\ell}}=-a$. Hence, 
\begin{align*}
\chi((b,k))=L_{(0,i_1)}^{p_1}\ldots L_{(0,i_s)}^{p_s}((b,k))=(b+\sum_{\ell=1}^{s} p_\ell\cdot c_{{i_\ell},k},k)=
(b+\sum_{\ell=1}^{s} p_\ell\cdot (c_{i_\ell-k} - c_{-k}),k).
\end{align*}
Since $m\mid k$ then, by Lemma~\ref{lem:5.3}, $c_{i_\ell}-c_{i_\ell-k}\in H$ and $c_{-k}\in H$. So, there is $\overline{h}\in H$ such that 
\[
\chi((b,k))=(b+\sum_{\ell=1}^{s} p_\ell\cdot c_{i_\ell}+\overline{h},k)=(b-a+\overline{h},k)
\asymp(b-a,k).
\]
Now, let $a,a'\in G$ and $i\neq i'\in\Z_n$ be such that $(a,i)\asymp(a',i')$. Then, $m\mid (i'-i)$ by Lemma~\ref{lem:5.3} and, by \eqref{i+k}, 
$(a,0)\asymp(a',i'-i)$. Hence, $(0,0)\asymp (a'-a,i'-i)$. Now
\[ (0,0)\asymp\textstyle (\pi^m\gamma)^{\frac{i'-i}m} ((0,0))
= \pi^{i'-i}\gamma^{\frac{i'-i}m}\iota^{(i'-i)/m \choose 2} ((0,0)) = (\frac{i'-i}m\cdot r+{(i'-i)/m \choose 2}\cdot h,i'-i).
\]
 Hence, once again by Lemma~\ref{lem:5.3}
 \[\textstyle a'-a-\frac{i'-i}m\cdot r+{(i'-i)/m \choose 2}\cdot h \in H\]
 which is equivalent to
 \[\textstyle a'-a+\frac{i'-i}m\cdot r\equiv 0\pmod H.\]
 
Conversely, we will show that an equivalence relation $\asymp$ defined by \eqref{def:con} is a congruence of the solution $(X,\sigma, \tau)$. Suppose $m<\infty$.

Let $(a,i)\asymp(a',i')$ and $(b,j)\asymp(b',j')$. By definition, $(j'+1)\equiv (j+1) \pmod m$ and  $b-b'+\tfrac{j'-j}m\cdot r\equiv 0\pmod H $. Further, by Lemma~\ref{lem:5.3},
\begin{multline*}
b-b'+\tfrac{j'-j}m\cdot r=
b+c_{i-j-1}-c_{-j-1}
  -b'-c_{i-j-1}+c_{-j-1}+\tfrac{j'-j}m\cdot r\equiv\\
b+c_{i-j-1}-c_{-j-1}-b'-c_{i'-j'-1}+c_{-j'-1}+\tfrac{j'-j}m\cdot r=\\
  b+c_{i-j-1}-c_{-j-1}-b'+c_{i'-j'-1}-c_{-j'-1}+\tfrac{(j'+1)-(j+1)}m\cdot r\equiv
0\pmod H. 
\end{multline*}
Since
 \begin{align*}
   \sigma_{(a,i)}((b,j)) &= (b+c_{i-j-1}-c_{-j-1},j+1), \quad {\rm and}\\
   \sigma_{(a',i')}((b',j')) &= (b'+c_{i'-j'-1}-c_{-j'-1},j'+1),
 \end{align*}
this finishes the proof for $m<\infty$. But the case $m=\infty$ is similar.
\qedhere
\end{proof}

Let $\theta(m,H,r)$ denote the congruence relation of a solution $\mathcal{S}(G\times \Z_n,\mathbf{c})$ specified  by the triple $m$, $H$ and $r$ from Proposition \ref{prop:image}. By \eqref{def:con}, $\theta(m,G,r)=\theta(m,G,s)$, whenever $r\equiv s\pmod H$. In Theorem \ref{thm:congr} we will show that
this is the only possibility how to obtain isomorphic solutions.



\begin{example}\label{exm:contr}
Let us consider the $12$-element solution $\mathcal{S}(\Z_2\times\Z_6,\mathbf{c}=(0,1,1,0,1,1))$. Since, $c_i=c_{i+3}$, for each $i\in\Z_6$, by Proposition \ref{prop:image} there are seven different congruences of the solution: 
$\theta(6,\{0\},0)$, $\theta(3,\{0\},0)$,  $\theta(3,\{0\},1)$, $\theta(6,\Z_2,0)$, $\theta(3,\Z_2,0)$, $\theta(2,\Z_2,0)$ and $\theta(1,\Z_2,0)$

The congruence $\theta=\theta(3,\{0\},1)$ induces the quotient solution 
$(Y,\sigma, \tau)$, where $Y=\{\mathbf{0},\mathbf{1},\mathbf{2},\mathbf{3},\mathbf{4},\mathbf{5}\}$ is defined by:
    \begin{align*}
    \mathbf{0}=(0,0)/_{\theta}=(1,3)/_{\theta}& & 
    \mathbf{1}=(1,0)/_{\theta}=(0,3)/_{\theta}& & 
    \mathbf{2}=(0,1)/_{\theta}=(1,4)/_{\theta},\\
    \mathbf{3}=(1,1)/_{\theta}=(0,4)/_{\theta}&& 
    \mathbf{4}=(0,2)/_{\theta}=(1,5)/_{\theta}& & 
    \mathbf{5}=(1,2)/_{\theta}=(0,5)/_{\theta}.
  \end{align*}
  It is then possible to compute the permutations:
  \begin{align*}
    \sigma_{\mathbf{0}}=\sigma_{\mathbf{1}}&=(\mathbf{0}\mathbf{2}\mathbf{4}\mathbf{1}\mathbf{3}\mathbf{5});\quad
    \sigma_{\mathbf{2}}=\sigma_{\mathbf{3}}=(\mathbf{0}\mathbf{3}\mathbf{5}\mathbf{1}\mathbf{2}\mathbf{4});\quad
    \sigma_{\mathbf{4}}=\sigma_{\mathbf{5}}=(\mathbf{0}\mathbf{2}\mathbf{5}\mathbf{1}\mathbf{3}\mathbf{4}).
  \end{align*}
 Hence $\mathcal{G}(Y)$ is a non-abelian group of order~24, $\dis{Y}=\{\mathrm{id},(\mathbf{2}\mathbf{3})(\mathbf{4}\mathbf{5}),(\mathbf{0}\mathbf{1})(\mathbf{4}\mathbf{5}),(\mathbf{2}\mathbf{3})(\mathbf{0}\mathbf{1})\}$ and $\dis{Y}\cap\langle\sigma_{\mathbf{0}}\rangle=\{\id\}$ with $|\dis{Y}_{\mathbf{0}}|=2$. But the solution $(Y,\sigma,\tau)$ is not isomorphic to a solution $\mathcal{S}(G\times\Z_n,\mathbf{c})$ for any $(G,+,0)$, $n$ and $\mathbf{c}$.
 Indeed, 
 by Proposition~\ref{prop:dis00}, such a group $(G,+,0)$ 
 would be a two-element group and,
 by Proposition \ref{prop:lengthcycles}, $n$ would be equal to~$6$
 and we would obtain a 12-element solution.
 
 \end{example}

\begin{example}
By Proposition \ref{prop:homfin}, the solution $(\Z_{2^{m}},\sigma,\tau)$ from Example \ref{exm:Rump} Case 1 is isomorphic to the solution 
 $\mathcal{S}(2\Z_{2^{m}}\times \Z_2,\mathbf{c}=(0,-2))$. The isomorphism $\Phi\colon 
2\Z_{2^{m}}\times \Z_2\to \Z_{2^{m}}$ is given by $(s,0)\mapsto s$ and $(s,1)\mapsto s+1$, and clearly $\kker\Phi=\theta(2,\{0\},0)$.

On the other hand, the solution $(\Z_{2^m},\sigma,\tau)$ from Case 2 is the homomorphic image of the solution $\mathcal{S}(2\Z_{2^{m}}\times \Z_4,\mathbf{c}=(0,2^{m-1}-2,0,2^{m-1}-2))$. The homomorphism $\Phi\colon 
2\Z_{2^{m}}\times \Z_4\to \Z_{2^{m}}$ is given by $(s,0)\mapsto s$, $(s,1)\mapsto s+1+2^{m-1}$, $(s,2)\mapsto s+2^{m-1}$ and $(s,3)\mapsto s+1$. In this case, $\kker\Phi=\theta(2,\{0\},2^{m-1})$.
\end{example}

By Proposition \ref{prop:image}, for each indecomposable solution of multipermutation level~$2$ there are a solution 
$(X,\sigma,\tau)=\mathcal{S}(G\times \Z_n,\mathbf{c})$, a subgroup $H$ of the group $(G,+,0)$, a number $m$, an element $r\in G$ and a congruence $\theta=\theta(m,H,r)$ determined by \eqref{def:con} such that the solution is isomorphic to the quotient solution of $\mathcal{S}(G\times \Z_n,\mathbf{c})$ by the relation $\theta$. Hence, 
each indecomposable solution of multipermutation level~$2$ is of the form $(X,\sigma,\tau)/\theta=(X/\theta,\obrazS{\sigma}{\theta},\obrazS{\tau}{\theta})$. 

Let $Y=X/\theta$. Clearly, for each $y\in Y$, there is $(a,i)\in G\times \Z_n$ such that $y=(a,i)/\theta$. In particular, for $y_1=(a_1,i_1)/\theta$ and $ y_2=(a_2,i_2)/\theta\in Y$, the equality $y_1=y_2$ implies $i_1-i_2\equiv 0\pmod m$ and, for $i\in \mathbb{N}$, the element $\widetilde{i}\in Y$ corresponds to the congruence class $\pi^i(\widetilde{0})/\theta=(0,i)/\theta$.

In the sequel (until Proposition~\ref{prop:Auttrans}) we suppose the following situation: we have $(Y,\rho,\upsilon)$, an indecomposable
solution of multipermutation level~$2$ and we know, according to
Propositions~\ref{prop:homim} and~\ref{prop:homfin}, that there exist
a~solution $(X,\sigma,\tau)=\mathcal{S}(G\times \Z_n,\mathbf{c})$
and a~congruence~$\theta=\theta(m,H,r)$ such that
$(Y,\rho,\upsilon)\cong (X/\theta,\obrazS{\sigma}{\theta},\obrazS{\tau}{\theta})$.

As the first application of the characterization of congruences of a solution $\mathcal{S}(G\times \Z_n,\mathbf{c})$ we prove the analogy of
Proposition~\ref{prop:lengthcycles} for all indecomposable solutions.

\begin{proposition}
Let $(Y,\rho,\upsilon)$ be a finite indecomposable solution of multipermutation level~$2$. Each permutation $\rho_y$, for $y\in Y$, is composed of cycles of length~$\ell$.
\end{proposition}

\begin{proof}
We shall prove that $\ell=km$,
 where $k$ is the smallest
 positive integer such that~$k\cdot r\in H$.
 
 Let $(a,i),(b,j)\in G\times \Z_n$ and we shall compute the smallest $\ell$ such that $\sigma_{(a,i)}^{\ell}((b,j))/\theta= (b,j)/\theta$.
 By \eqref{sigmal},
 we have
 \[\sigma_{(a,i)}^{\ell}((b,j)) = (b+\sum_{j'=1}^{\ell} (c_{i-j-j'}-c_{-j-j'} )\ ,\ j+\ell)\]
 and therefore $m\mid \ell$. Moreover, since $c_{i'}-c_{i'+\ell}\in H$, for each $i'$, this implies that there exists $h\in H$ such that
 $\sum_{j'=1}^{\ell} c_{i-j-j'}-c_{-j-j'}=h$.
 Now from $\sigma_{(a,i)}^{\ell}((b,j))/\theta= (b,j)/\theta$ we obtain
 \[b+h-b+\tfrac{j-j-\ell}m \cdot r\in H\]
 which is equivalent to $\frac \ell m\cdot r\in H$. 
\end{proof}

In the sequel we shall show several properties
of the image determined by the parameters of the congruence.

\begin{proposition}\label{prop:510}
 Let $(Y,\rho,\upsilon)$ be an indecomposable solution of multipermutation level~$2$ and 
~$e\in Y$.
 Then $\dis{Y}/\dis{Y}_e\cong G/H$.
\end{proposition}

\begin{proof}
We can suppose, without loss of generality, that $e=(0,0)/\theta$. Clearly, each permutation $\delta\in \mathcal{G}(X)$ induces the permutation $\obrazS{\delta}{\theta}\in  \mathcal{G}(Y)$. In particular, $\obrazS{L}{\theta}_{(a,i)/\theta}((b,j)/\theta)=L_{(a,i)}((b,j))/\theta$ and, for each $\gamma\in \dis{X}_{(0,0)}$, $\obrazS{\gamma}{\theta}(e)=\obrazS{\gamma}{\theta}((0,0)/\theta)=\gamma((0,0))/\theta=(0,0)/\theta=e$. Hence, the mapping $\gamma\mapsto \obrazS{\gamma}{\theta}$ is a well defined homomorphism from $\dis{X}/\dis{X}_{(0,0)}$ onto $\dis{Y}/\dis{Y}_e$. 

By Remark \ref{prop:dis00} we know that  $\Psi\colon \dis{X}\to G$  defined by $\prod L_{(0,i)}^{p_i}\mapsto \sum p_i\cdot c_i$ is a well-defined epimorphism with $\Ker\Psi=\dis{X}_{(0,0)}$ and $G\cong \dis{X}/\dis{X}_{(0,0)}$. Hence there exists a homomorphism $\Phi$
 from~$(G,+,0)$ onto~$\dis{Y}/\dis{Y}_{e}$.
 Let us now compute $\Ker\Phi$. An element
$\prod \obrazS{L}{\theta}^{p_k}_{(a_k,i_k)/\theta}$ lies in $\dis{Y}_e$ if and only if $\prod L^{p_k}_{(a_k,i_k)}((0,0))=(h,0)$, for some~$h\in H$. Since $\prod L^{p_k}_{(a_k,i_k)}((0,0))=(\sum p_k\cdot c_{i_k},0)$,
 this is equivalent to $\sum p_k\cdot c_{i_k}\in H$
 and therefore also to $\Psi(\prod L^{p_k}_{(a_k,i_k)})\in H$. Hence $\Ker\Phi=H$, which finishes the proof.
\end{proof}

\begin{lemma}
 Let $(G,\cdot,1)$ be a group that acts on a set~$X$. Then $\langle \bigcup G_x\rangle$
 is a normal subgroup of~$(G,\cdot,1)$.
\end{lemma}

\begin{proof}
 Let $g\in G_x$, for some~$x\in X$, and
 let $h\in G$. Then $hgh^{-1}(h(x))=hg(x)=h(x)$ and therefore $hgh^{-1}\in G_{h(x)}$.
\end{proof}

\begin{lemma}
Let 
$(Y,\rho,\upsilon)$ be an indecomposable solution of multipermutation level~$2$ and let~$e\in Y$. Then
 $\mathcal{G}(Y)_e\dis{Y}=\langle \bigcup \mathcal{G}(Y)_y\rangle\dis{Y}$
 and it is a normal subgroup of~$\mathcal{G}(Y)$.
\end{lemma}

\begin{proof}
 Let $\alpha\in \mathcal{G}(Y)_y$, for some~$y\in Y$. Let $\beta\in \mathcal{G}(Y)$ send $e$ to~$y$. Then $\beta^{-1}\alpha\beta\in \mathcal{G}(Y)_e$
 and $\alpha=\beta^{-1}\alpha\beta[\beta,\alpha]\in\mathcal{G}(Y)_e\dis{Y}$.
 Hence $\bigcup \mathcal{G}(Y)_y\subseteq \mathcal{G}(Y)_e\dis{Y}$.
 
 Now $\langle \bigcup \mathcal{G}(Y)_y\rangle\dis{Y}$ is normal since it is a product of two normal subgroups.
\end{proof}

\begin{proposition}\label{prop:512}
Let $(Y,\rho,\upsilon)$ be an indecomposable solution of multipermutation level~$2$ and let $e\in Y$. Then
 $\mathcal{G}(Y)/(\mathcal{G}(Y)_e\dis{Y})\cong \Z_m$.
\end{proposition}

\begin{proof}
Let $\Phi\colon X\to Y$ be the natural projection from $(X,\sigma, \tau)$ onto $(Y,\rho,\upsilon)$.
 This homomorphism induces the onto group homomorphism $\Phi_{\mathcal{G}}\colon \mathcal{G}(X)\to\mathcal{G}(Y)$ with $\Phi_{\mathcal{G}}(\nu)(x/\theta)=\nu(x)/\theta$, for $\nu\in \mathcal{G}(X)$. Further, let $\Psi$ be the natural projection from the group $\mathcal{G}(Y)$ onto 
$\mathcal{G}(Y)/(\langle \bigcup \mathcal{G}(Y)_y\rangle\dis{Y})$.
 This means that $\Ker\Psi=\langle \bigcup \mathcal{G}(Y)_y\rangle\dis{Y}$
 and $\mathcal{G}(Y)/\Ker\Psi\cong \mathcal{G}(X)/\Phi^{-1}_{\mathcal{G}}[\Ker\Psi]$. 
 
Recall that $m$ is the smallest positive integer such that there exists $a\in G$ with
$(0,0)/\theta=(a,m)/\theta$.
Clearly, $\Phi_{\mathcal{G}}(\dis{X})\subseteq\Ker\Psi$. Now let $\gamma\in\dis{X}$ be such that
 $\gamma((0,m))=(a,m)$. We then know
 that $\Phi_{\mathcal{G}}(\gamma\pi^m)\in \mathcal{G}(Y)_{\Phi((0,0))}$. Hence
 $\Phi_{\mathcal{G}}(\gamma\pi^m)\in\Ker\Psi$
 and therefore $\dis{X}\langle\pi^m\rangle\subseteq \Phi_{\mathcal{G}}^{-1}[\Ker\Psi]$. On the other hand, let $\gamma\pi^k\in \mathcal{G}(X)$ and $\Phi_{\mathcal{G}}(\gamma\pi^k)\in {\mathcal{G}}_{\Phi((0,0))}$. Hence $(0,0)/\theta=\Phi((0,0))=\Phi_{\mathcal{G}}(\gamma\pi^k)(\Phi((0,0)))=\Phi(\gamma\pi^k((0,0)))=\gamma\pi^k((0,0))/\theta=\gamma((0,k))/\theta$. So, $(0,0)/\theta=\gamma((0,k))/\theta$ and $m|k$. Therefore, $\pi^k\in \langle\pi^m\rangle$ which implies $\Phi_{\mathcal{G}}^{-1}[\Ker\Psi]\subseteq \dis{X}\langle\pi^m\rangle$.
 
Finally, $\mathcal{G}(Y)/\Ker\Psi\cong \mathcal{G}(X)/\Phi^{-1}_{\mathcal{G}}[\Ker\Psi]=\dis{X}\langle\pi\rangle/\dis{X}\langle\pi^m\rangle=\langle\pi\rangle/\langle\pi^m\rangle
\cong\Z_m$.
\end{proof}

\begin{corollary}\label{cor:amount}
Let $(Y,\rho,\upsilon)$ be an indecomposable solution of multipermutation level~$2$. Then $|Y|=m\cdot[G:H]$.
\end{corollary}

\begin{proof}
 Choose $e\in Y$ and
 \begin{multline*}
  |Y|=[\mathcal{G}(Y):\mathcal{G}(Y)_e]=
  [\mathcal{G}(Y):\mathcal{G}(Y)_e\dis{Y}]\cdot [\mathcal{G}(Y)_e\dis{Y}:\mathcal{G}(Y)_e]=\\
  [\mathcal{G}(Y):\mathcal{G}(Y)_e\dis{Y}]\cdot [\dis{Y}:\dis{Y}_e]=m\cdot [G:H]. \qedhere
 \end{multline*}
\end{proof}

\begin{lemma}\label{lem:eq}
Let $(Y,\rho,\upsilon)$ be an indecomposable solution of multipermutation level~$2$. For each~$x,y\in Y$, $\rho_x^m=\rho_y^{m}$.
\end{lemma}

\begin{proof}
By Proposition \ref{prop:image}, for each~$i\in\Z_n$, there exists some $h\in H$, such that, for $(a,i),(b,j)\in G\times \Z_n$,
 \[\sigma_{(a,i)}^m((b,j))=(b+\sum_{\ell=1}^m(c_{i-j-\ell}-c_{-j-\ell}),j+m)=(b+h,j+m).\]
 Since $\sigma_{(0,0)}^m((b,j))=(b,j+m)$, by \eqref{def:con} we have that $\sigma_{(a,i)}^m((b,j))/\theta =\sigma_{(0,0)}^m((b,j))/\theta$, which means
 that, for all $y\in Y$, $\rho_y^m=\rho_{(a,i)/\theta}^m=\rho_{(0,0)/\theta}^m$.
\end{proof}


\begin{proposition}\label{prop:Auttrans}
Let $(Y,\rho,\upsilon)$ be an indecomposable solution of multipermutation level~$2$. The automorphism group $\aut{Y}$ is regular.
\end{proposition}

\begin{proof}
Let $\gamma\in\dis{Y}$, $\pi=\rho_{(0,0)/\theta}$ and $k\in\N$.
 For each $i\in\N$ and $\alpha\in\dis{Y}$, let
 \[\Phi(\alpha(\widetilde{i}))=
 \rho_{\widetilde k}^i \pi^{k-i}\alpha\pi^{i-k}\gamma(\widetilde k).
 \]
We want to show that $\Phi$ is an automorphism of the solution $(Y,\rho,\upsilon)$
but first we have to show that~$\Phi$ is well defined
since each element $\alpha(\widetilde i)$ may admit several representations.
Let $\alpha,\alpha'\in\dis{Y}$ and $i,i'\in \mathbb{N}$. First note that  $\pi^{k-i}\alpha\pi^{i-k}\gamma\in\dis{Y}$. Further, if $i-i'=m\ell$, for some $\ell\in \mathbb{N}$,  then by Lemma \ref{lem:eq} we have
 \begin{multline*}
\alpha(\widetilde i)=\alpha'(\widetilde{i'})\quad \Leftrightarrow\quad \alpha\pi^{i-k}(\widetilde k)=\alpha'\pi^{i'-k}(\widetilde{k})
\quad \Leftrightarrow\quad 
\gamma\pi^{k-i} \alpha\pi^{i-k}(\widetilde k)=\gamma\pi^{-m\ell}\pi^{k-i'}\alpha'\pi^{i'-k}(\widetilde{k})\quad \Leftrightarrow\quad\\
\rho^i_{\widetilde{k}}\pi^{k-i} \alpha\pi^{i-k}\gamma(\widetilde k)=\rho^i_{\widetilde{k}}\rho_{\widetilde k}^{-m\ell}\pi^{k-i'}\alpha'\pi^{i'-k}\gamma(\widetilde{k})
\quad \Leftrightarrow\quad
\rho^i_{\widetilde{k}}\pi^{k-i} \alpha\pi^{i-k}\gamma(\widetilde k)=\rho^{i'}_{\widetilde{k}}\pi^{k-i'}\alpha'\pi^{i'-k}\gamma(\widetilde{k})
\quad \Leftrightarrow\quad\\
\Phi(\alpha(\widetilde{i}))=\Phi(\alpha'(\widetilde{i'})).
 \end{multline*}
Hence, the mapping $\Phi$ is well-defined 
 since $\alpha(\widetilde i)=\alpha'(\widetilde{i'})$ 
 implies $m\mid (i-i')$.
%
Conversely, 
 necessarily $\rho_{\widetilde k}^i \pi^{k-i}\alpha\pi^{i-k}\gamma(\widetilde k)=
\rho_{\widetilde k}^{i'} \pi^{k-i'}\alpha'\pi^{i'-k}\gamma(\widetilde k)$
 implies $m\mid (i-i')$,
 and therefore $\Phi$ is injective.
 
 Moreover, $\Phi(\widetilde 0)=\gamma(\widetilde k)$ and $\Phi$ is clearly onto. We shall now prove that it is endomorphism of the solution. Note that $\rho^{i}_{\widetilde{k}}\gamma\pi^{k-i}\alpha\pi^{-k}=\rho^{i}_{\widetilde{k}}\pi^{-i}\pi^k\alpha\pi^{-k}\gamma\in\dis{Y}$. This implies, for $\beta\in \dis{Y}$,
 
 \begin{multline*}
  \Phi(\rho_{\alpha(\widetilde i)}(\beta(\widetilde j)))  \stackrel{\eqref{eq:disX}}{=}
  \Phi ({L}_{\widetilde i}\pi\beta(\widetilde j)) =
  \Phi ({L}_{\widetilde i}\pi\beta\pi^{-1}(\widetilde {j+1}))\\ 
  =\rho_{\widetilde k}^{j+1}\pi^{k-j-1}{L}_{\widetilde i}\pi\beta\pi^{-1}\pi^{1+j-k}\gamma(\widetilde k)
  \stackrel{\eqref{piLxpi}}{=}\rho_{\widetilde k}^{j+1}{L}_{\widetilde{k-j-1}}^{-1}{L}_{\widetilde{i+k-j-1}}\pi^{k-j-1}\pi\beta\pi^{j-k}\gamma(\widetilde k)\\
  \stackrel{\eqref{piLxpi}}{=}{L}_{\widetilde{k}}^{-1}{L}_{\widetilde{i+k}}
  \rho_{\widetilde k}^{j+1}
  \pi^{k-j}\beta\pi^{j-k}\gamma(\widetilde k)
  ={L}_{\widetilde{i+k}}\pi
  \rho_{\widetilde k}^{j}
  \pi^{k-j}\beta\pi^{j-k}\gamma(\widetilde k)\\=
  \rho_{\widetilde{i+k}} \Phi(\beta(\widetilde j))
  \stackrel{\eqref{eq:disX}}{=}
  \rho_{\rho_{\widetilde k}^i\pi^{k-i}\alpha\pi^{i-k}\gamma(\widetilde{k})} \Phi(\beta(\widetilde j))=
  \rho_{\Phi(\alpha(\widetilde i))}
  \Phi(\beta(\widetilde j))
 \end{multline*}
and hence $\aut{Y}$ is transitive. And, according to Corollary~\ref{cor:semi-regular}, $\Aut(Y)$ is semi-regular.
\end{proof}

\begin{theorem}\label{thm:congr}
 Let $(X,\sigma,\tau)=\mathcal{S}(G\times \Z_n,\mathbf{c})$.
Let $\theta_1=\theta(m_1,H_1,r_1)$ and $\theta_2=\theta(m_2,H_2,r_2)$. Then
the solutions $(X,\sigma,\tau)/\theta_1$ and $(X,\sigma,\tau)/\theta_2$
are isomorphic if and only if $m_1=m_2$, $H_1=H_2$ and $r_1\equiv r_2\pmod{H_2}$.
\end{theorem}

\begin{proof}
 ``$\Rightarrow$''  Suppose that there is an isomorphism $\Phi:X/\theta_1\to X/\theta_2$. Clearly $m_1=m_2$, according
 to Proposition~\ref{prop:510}.
 According to Proposition~\ref{prop:Auttrans},
 we can suppose, without loss of generality $\Phi((0,0)/\theta_1)=(0,0)/\theta_2$. Since $\sigma^i_{(0,0)}((0,0))=(0,i)$,
 necessarily $\Phi((0,i)/\theta_1)=(0,i)/\theta_2$, for each~$i
 \in\Z$. Now, for each~$a\in G$ and $i\in\Z_n$, there exist $j_1,\ldots,j_k\in \Z_n$ and $p_1,\ldots,p_k\in \Z$  such that
 \[(a,i)=\prod L_{(0,j_\ell)}^{p_\ell}((0,i))=\prod(\sigma_{(0,j_\ell)}\sigma_{(0,0)}^{-1})^{p_\ell}((0,i))\]
 and therefore $\Phi((a,i)/\theta_1)=(a,i)/\theta_2$. 
 
 Suppose~$h\in H_1$. Then by \eqref{def:con}, 
$(h,0)/\theta_2=\Phi((h,0)/\theta_1)=\Phi((0,0)/\theta_1)=(0,0)/\theta_2$ 
 and therefore $h\in H_2$ and $H_1\subseteq H_2$. Analogously
 $H_2\subseteq H_1$.
 Now, once again by \eqref{def:con}
 \[(r_2,m_2)/\theta_2=(0,0)/\theta_2=\Phi((0,0)/\theta_1)=\Phi((r_1,m_2)/\theta_1)=(r_1,m_2)/\theta_2,
 \]
which implies $r_1\equiv r_2\pmod {H_2}$.
 
 ``$\Leftarrow$'' If $m_1=m_2$, $H_1=H_2$ and $r_1\equiv r_2\pmod {H_2}$ then $\theta_1=\theta_2$.
\end{proof}

\section{Enumeration}

We have now a theoretic means of constructing all indecomposable solutions
of multipermutation level~$2$
and we shall explicitly do so in some particular cases.
Every such solution
is a homomorphic image of the solution $\mathcal{S}((\bigoplus_\Z \Z)\times \Z,\mathbf{c})$ from Proposition~\ref{prop:homim}. To construct all of them we have to find  all congruences $\theta(m,H,r)$ of the solution. 
Then, according to Theorem~\ref{thm:congr}, different congruences give non-isomorphic solutions.
If~$m$ is finite and $H$ is a subgroup of $(G,+,0)=\bigoplus_\Z\Z$ then, by Proposition~\ref{prop:image} (iii),
\[0=c_0\equiv c_m=\sum_{i=1}^m e_i\pmod H\]
and also, for each $i\in\Z$,
$ e_{i+m} = c_{i+m}-c_{i+m-1} \equiv 
 c_i-c_{i-1} =e_i\pmod H.
$ 
In particular, the group $G/H$ has at most~$m-1$ generators.

If we want to construct all finite $n$-element indecomposable solutions of multipermutation level~$2$ necessarily, by Corollary \ref{cor:amount}, we have to consider all $m$ which divide~$n$. Moreover we know $[G:H]=\frac{n}{m}$. 

If $m=1$ then $G/H$ has~$0$ generators and hence $G=H$.
Another trivial case is $m=n$; then $G=H$ and $\mathcal{S}((\bigoplus_\Z \Z)\times \Z,\mathbf{c})/\theta(m,G,0)$ is a permutation solution.

If $m=2$ then $G/H$ has to be cyclic and there is only one choice
of such~$H$. Nevertheless, there are $[G:H]$ choices of~$r$ (one for each coset of~$H$) and hence there are $[G:H]$ non-isomorphic solutions.


Another easy to calculate case is the case when $G/H$ is
essentially a vector space. For this situation,
let us recall that, given a vector space $V$ of dimension~$n$ over a field~$\mathbb{F}_q$, the number of all subspaces of ~$V$ of dimension~$0<k\leq n$ is equal
  to
\begin{align}\label{num:vect}
\frac{(q^n-1)(q^n-q)\cdots(q^n-q^{k-1})}{(q^k-1)(q^k-q)\cdots(q^k-q^{k-1})}
=
\frac{(q^n-1)(q^{n-1}-1)\cdots(q^{n-k+1}-1)}{(q^k-1)(q^{k-1}-1)\cdots(q-1)}.
\end{align}



\begin{proposition}\label{number}
  Let~$p$ be a prime number, $(A,+,0)$ be an elementary abelian group of order~$p^k$ and let~$m\in\N$. Then the number $\mathfrak{n}(A,m)$ 
of indecomposable solutions $(X,\sigma,\tau)$ of multipermutation
  level~$2$ such that $\dis{X}/\dis{X}_e\cong A$ and  $[\mathcal{G}(X):\mathcal{G}(X)_e\dis{X}]=m$ is
  \begin{align*}
  \mathfrak{n}(A,m)=\begin{cases}
      p^k\cdot\frac{(p^{m-1}-1)(p^{m-2}-1)\cdots(p^{k+1}-1)}{(p^{m-k-1}-1)(p^{m-k-2}-1)\cdots(p-1)} &\text{if }m>k+1,\\
      p^k & \text{if }m=k+1,\\
      0 & \text{if }m\leq k.
      \end{cases}
    \end{align*}  
\end{proposition}

\begin{proof}
  Every indecomposable solution $(X,\sigma,\tau)$ of multipermutation
  level~$2$ can be obtained as a quotient of the solution  $\mathcal{S}(G,\mathbf{c})$ with $G=(\bigoplus_\Z \Z)\times \Z$ via a congruence
  $\theta(m,H,r)$, for some subgroup~$H$ with $G/H\cong A$ and $r\in G$. According to Theorem~\ref{thm:congr}, different choices of~$H$ and $r\pmod H$
  yield non-isomorphic solutions. Let us count the number of such choices of~$H$.
  
  According to Proposition~\ref{prop:image}, $\sum_{i=1}^m e^i\in H$ and also $e_{i+m}-e_i\in H$, for each~$i\in\Z$. Let now
  \[N=\left\langle \sum_{i=1}^m e^i,\ e_{i+m}-e_i,\ p\cdot e_i \mid i\in \Z\right\rangle.\]
  Clearly, $G/N$ is an elementary abelian group of order~$p^{m-1}$ and $N\leq H$. If we treat $G/N$ as a vector space over~$\mathbb{F}_p$ of dimension~$m-1$ then $H/N$ is a subspace of dimension~$m-1-k$. According to \eqref{num:vect}, for $k<m-1$, there are 
  \[\frac{(p^{m-1}-1)(p^{m-2}-1)\cdots(p^{k+1}-1)}{(p^{m-k-1}-1)(p^{m-k-2}-1)\cdots(p-1)}\]
  such choices of~$H/N$ and therefore as many choices of~$H$. If $k=m-1$ there is exactly 1 subspace of dimension~$0$. If $k>m-1$ then there is no such subspace.
  
  Now, for each choice of~$H$, there are $p^k$ cosets of~$H$ and therefore there are~$p^k$ choices of~$r$.
\end{proof}

\begin{remark}\label{number2}
  For $k=1$, the expression for the number $\mathfrak{n}(A,m)$ simplifies to $\dfrac{p^m-p}{p-1}$.
\end{remark}

\begin{remark}
By Proposition \ref{number} there are exactly 
$ 1+\frac{p^{p}-p}{p-1}$ 
indecomposable solutions of multipermutation level~$2$ of order $p^2$, for a prime $p$. Exactly $p^{p-1}$ of them are composed of cycles of length~$p^2$. For example, there are $13$ indecomposable solutions of multipermutation level~$2$ of order $9$ and $9$ of them are cycles of length~$9$.

Let $2<p<q$ be two different primes.
There are $1+\frac{p^{q}-p}{p-1}
 +\frac{q^{p}-q}{q-1}$
 indecomposable solutions of multipermutation level~$2$ of order $pq$. For example, there are $151$ indecomposable solutions of multipermutation level~$2$ of order $15$.
Or, another example, there are
 $2^{p}+p-1$ indecomposable solutions of multipermutation level~$2$ of order $2p>4$. 
\end{remark}
Since the work of Etingof, Schedler and Soloviev~\cite{ESS} the
researchers ask themselves a question about the proportion of the numbers of the indecomposable solutions with respect to the decomposable ones.
And it was conjectured by Vendramin that the limit towards infinity
of the ratio indecomposables/decomposables goes to 0.
The result of Blackburn~\cite{B13} about the numbers of racks can be easily adapted
to solutions and therefore we know that there are at least $2^{O(n^2)}$ decomposable solutions for each size~$n$.
Now it turns out that there are exponentially many indecomposable
solutions as well.

\begin{corollary}
 Let~$s\in\N$. Then there are  at least
 $2^{n/2}-1$ indecomposable solutions of size~$n=2^s$.
\end{corollary}

\begin{proof}
 For $s=1$ there exists exactly one solution, namely the permutation one. Hence suppose $s>1$. 
 According to Proposition~\ref{number} and Remark~\ref{number2}, there are $2^{2^{s-1}}-2$ solutions of size~$2^s$ with $m=2^{s-1}$.
 Plus there is one solution with $m=2^s$.
\end{proof}

Regardless of the exponential growth, still $2^{O(n)}\ll 2^{O(n^2)}$. Of course, we have a lower bound for the number of indecomposables only but the experience from other types of algebraic structures suggests that the upper bound might be similar.

\begin{conjecture}
 There is a constant $c\in\mathbb{Q}
^+$ such that the number of indecomposable
 solutions of size~$n$ is less then $2^{cn}$, for each~$n\in \N$.
\end{conjecture}

Proposition~\ref{number} may be generalized to the case
when $A$ is a free $\Z_t$-module, for some number~$t$. However, the expression would be too complicated and hence we give a special
case only, namely with a group $(A,+,0)$ cyclic of size~$p^k$.

\begin{proposition}\label{number3}
  Let~$p$ be a prime number, let $(A,+,0)$ be a cyclic group of order~$p^k$ and let~$m>1$. Then the number $\mathfrak{n}(A,m)$ 
of indecomposable solutions $(X,\sigma,\tau)$ of multipermutation
  level~$2$ such that $\dis{X}/\dis{X}_e\cong A$ and  $[\mathcal{G}(X):\mathcal{G}(X)_e\dis{X}]=m$ is
  \[ \mathfrak{n}(A,m)=p^{km-m-k+2}\cdot\frac{p^{m-1}-1}{p-1}.\]
\end{proposition}

\begin{proof}
 The proof is the same as for Proposition~\ref{number}, with the exception that $N$ is a free $\Z_{p^k}$-module of rank $m-1$ and we
 count the number of its hyperplanes (free submodules of rank $m-2$).
 It is well known that each hyperplane corresponds to a line in the dual free module and hence the number of hyperplanes is equal to the number of lines. Each line is generated by a vector with at least one coordinate invertible. There are $(p^k)^{m-1}-(p^{k-1})^{m-1}$ such vectors. There are $p^k-p^{k-1}$ multiples of the vector
 that generate the same submodule. Hence there are
 $\frac{(p^k)^{m-1}-(p^{k-1})^{m-1}}{p^k-p^{k-1}}$ choices of~$H$.
 For each such~$H$, there are $p^k$ choices of~$r$.
\end{proof}

\begin{example}
 Let us compute the number of all indecomposable solutions of multipermutation level~$2$ of size 16:
 \begin{itemize}
  \item $8$ solutions with $m=2$ and $G/H\cong \Z_8$,
  \item $112$ solutions with $m=4$ and $G/H\cong \Z_4$,
  \item $28$ solutions with $m=4$ and $G/H\cong \Z_2^2$,
  \item $254$ solutions with $m=8$ and $G/H\cong \Z_2$,
  \item $1$ solution with $m=16$ and $G/H\cong \Z_1$.
 \end{itemize}
\end{example}
Analogously, just using the ideas from the beginning of the section and Propositions~\ref{number} and~\ref{number3}, we are able to compute the numbers
of indecomposable solutions of multipermutation level~2 up to the size of~$17$. We can also compare our numbers with
the results from \cite{AMV} (the numbers from the first three rows of
Table 1).
\begin{table}[h!]
$$\begin{array}{|r|r@{\>\;}r@{\>\;}r@{\>\;}r@{\>\;}r@{\>\;}r@{\>\;}r@{\>\;}r@{\>\;}r@{\>\;}r@{\>\;}r@{\>\;}r@{\>\;}r@{\>\;}r@{\>\;}r@{\>\;}r|}\hline
n                       & 1& 2& 3& 4&  5&  6&   7&    8&9&10&11&12&13&14&15&16\\\hline
\text{solutions}    & 1& 2& 5& 23&  88&  595&   3456&    34530&321931&4895272&&&&&& \\ \hline
\text{$2$-permut.}    & 1& 2& 5& 19&  70&  359&   2095&    16332&&&&&&&& \\ \hline
\text{indecom.}    & 1& 1& 1& 5& 1& 10&  1&   100&    16&36&1&&1&&& \\ \hline
\text{$2$-permut. ind.}  & 1& 1& 1& 3&  1&  10&   1&    19&13&36&1&136&1&134&151&403\\ \hline
\text{$2$-per. ind. abel.}  & 1& 1& 1& 3&  1&  1&   1&    3&4&1&1&3&1&1&1&7\\ \hline
\text{$2$-per. ind. cycl.}  & 1& 1& 1& 2&  1&  1&   1&    2&3&1&1&2&1&1&1&4\\ \hline
\end{array}$$
\caption{The number of $2$-permutational indecomposable solutions of size $n$, up to isomorphism.}
\label{Fig:count_solution}
\end{table}

Well known results of Etingof, Schedler, Soloviev \cite{ESS} and Etingof, Guralnick, Soloviev \cite{EGS} say that indecomposable solutions of prime cardinality are affine and have cyclic permutation group. Smoktunowicz and Smoktunowicz investigated in \cite{SS18} whether an analogous characterization for indecomposable solutions of arbitrary cardinality also exists and they concluded that with high probability to obtain similar results are not possible. But perhaps not all is lost and some analogies may be achieved.

In \cite[Proposition 6.1]{C21} Castelli showed that each uniconnected solution of square-free odd order is always of multipermutation level at most~$2$. On the other hand by Example \ref{exm:contr} there is $2$-permutational indecomposable solution of size $6$ which is not 
uniconnected. 
According to Table 1 the following natural question arises:
\begin{ques}
Is it true that each indecomposable solution of square-free order is always of multipermutation level at most~$2$?
\end{ques}



\begin{thebibliography}{99}

\bibitem{AMV}
Ö.~Agkün, M. Mereb, L. Vendramin, {\it Enumeration of set-theoretic solutions to the Yang-Baxter equation}, Math. Comp. {\bf 91} (2022), 1469--1481.

\bibitem{BCJ16} 
D. Bachiller, F. Cedó, E. Jespers, {\it Solutions of the Yang-Baxter equation associated with a left brace}, J. Algebra {\bf 463} (2016), 80--102.

\bibitem{B13}
S. Blackburn, \emph{Enumerating finite racks, quandles and kei},
Electron. J. Combin. {\bf 20} (2013), no. 3, Paper 43, 9 pp.


\bibitem{CCP}
M. Castelli, F. Catino, G. Pinto, {\it Indecomposable involutive set-theoretic solutions of the Yang-Baxter equation}, J. Pure Appl. Algebra {\bf 223} (2019), 4477--4493.

\bibitem{C21}
M. Castelli, {\it The construction of uniconnected involutive solutions of the
Yang-Baxter equation of multipermutational level 2 with odd
size and a Z-group permutation group}, available at {\tt http://arxiv.org/abs/2110.01912}

\bibitem{CGS17}
F. Ced\'{o}, T. Gateva-Ivanova, A. Smoktunowicz, {\it On the Yang–Baxter equation and left nilpotent left braces}, J. Pure Appl. Algebra {\bf 221} (2017), 751--756.

\bibitem{CJO}
F. Ced\'{o}, E. Jespers, J. Okni\'{n}ski, {\it Braces and the Yang-Baxter equation}, Comm. Math. Phys. {\bf 327} (2014), 101--116. Extended version arXiv:1205.3587.

\bibitem{CO21}
F. Ced\'{o}, J. Okni\'{n}ski, {\it Constructing finite simple solutions of the Yang-Baxter equation}, Adv. in Math. {\bf 391} (2021), 107968.

\bibitem{Dr90}
V.G. Drinfeld, {\it On some unsolved problems in quantum group theory}, In: P.P. Kulish (ed.) Quantum
groups, in: Lecture Notes in Math., vol. 1510, Springer-Verlag, Berlin, (1992), pp. 1--8.

\bibitem{ESS}
P. Etingof,  T. Schedler, A. Soloviev, {\it Set-theoretical solutions to the quantum Yang-Baxter equation}, Duke Math. J. {\bf 100} (1999), 169--209.

\bibitem{EGS}
P. Etingof, R. Guralnick, A. Soloviev, {\it Indecomposable set-theoretical solutions to the quantum Yang-Baxter equation on a set with a prime number of elements}, J. Algebra, {\bf 249} (2001), 709--719.

\bibitem{gap}
The GAP Group, GAP -- Groups, Algorithms, and Programming, Version 4.11.1; 2021. Available at {\tt http://www.gap-system.org}

\bibitem{GI18}
T. Gateva-Ivanova, {\it Set-theoretic solutions of the Yang-Baxter equation, braces and symmetric groups}, Adv. in Math. {\bf 338} (2018), 649--701.

\bibitem{GIM11}
T. Gateva-Ivanova, S. Majid, {\it Quantum spaces associated to multipermutation solutions of level two}, Algebr. Represent. Theory {\bf 14} (2011), 341--376.

\bibitem{GIC12}
T. Gateva-Ivanova, P. Cameron, {\it Multipermutation solutions of the Yang-Baxter equation}, Comm. Math. Phys. {\bf 309} (2012), 583--621.

\bibitem{HSV}
A. Hulpke, D. Stanovsk\'y, P. Vojt\v echovsk\'y, {\it Connected quandles and transitive groups}, J. Pure Appl. Algebra {\bf 220} (2016), 735--758.

\bibitem{JPZ20}
P. Jedli\v cka, A. Pilitowska, A. Zamojska-Dzienio, {\it The construction of multipermutation solutions of the Yang-Baxter equation of level 2}, J. Comb. Theory Ser. A. {\bf 176} (2020), 1--35.

\bibitem{JPZ21}
P. Jedli\v cka, A. Pilitowska, A. Zamojska-Dzienio, {\it Indecomposable involutive solutions of the Yang--Baxter equation of multipermutational level 2 with abelian permutation group}, Forum Math. {\bf 33}(5)  2021, 1083--1096.

\bibitem{JPZ22}
P. Jedli\v cka, A. Pilitowska, A. Zamojska-Dzienio, {\it Cocyclic braces and indecomposable cocyclic solutions of the Yang-Baxter equation}, to appear in Proc. Amer. Math. Soc. (2022).

\bibitem{Jimbo}
M. Jimbo, {\it Introduction to the Yang-Baxter equation}, Int. J. Modern Physics A, {\bf 4-15} (1989), 3759--3777.

\bibitem{J82}
D. Joyce, \emph{Classifying invariant of knots, the knot quandle}, J. Pure Applied Algebra, {\bf 23} (1982), 37--65.

\bibitem{Kassel}
C. Kassel, {\it Quantum groups}, Springer-Verlag, 2012.

\bibitem{Rump05}
W. Rump, {\it A decomposition theorem for square-free unitary solutions of the quantum
Yang-Baxter equation}, Adv. Math. {\bf 193} (2005), 40--55.

\bibitem{Rump07}
W. Rump, {\it Braces, radical rings, and the quantum Yang-Baxter equation}, J. Algebra {\bf 307} (2007), 153--170.


\bibitem{Rump20}
W. Rump, {\it Classification of indecomposable involutive set-theoretic solutions to the Yang-Baxter equation}, Forum Math. {\bf 32}(4) (2020), 891–903.

\bibitem{SS18}
A. Smoktunowicz, A. Smoktunowicz, {\it Set-theoretic solutions to the Yang-Baxter
equation and new classes of $R$-matrices}, Linear Algebra Appl. {\bf 546} (2018), 86–114.

\end{thebibliography}
\end{document}